\documentclass[11pt,a4paper,final]{amsart}
\usepackage[margin=1in]{geometry}
\usepackage[english]{babel}
\usepackage[latin1]{inputenc}
\usepackage[T1]{fontenc}
\usepackage[nobysame]{amsrefs}
\usepackage{graphicx}
\usepackage{subfigure}
\usepackage{amsmath}
\usepackage{amsfonts}
\usepackage{amssymb}
\usepackage{amsthm}
\newcommand{\R}{\mathbb{R}}  
\newcommand{\N}{\mathbb{N}}  
\newcommand{\PP}{\mathbb{P}}  
\newcommand{\E}{\mathbb{E}}  
\newcommand{\norm}[1]{\lVert#1\rVert}
\newcommand{\abs}[1]{\left\vert #1\right\vert}
\DeclareMathOperator{\tr}{tr}
\DeclareMathOperator*{\argmin}{arg\,min}
\DeclareMathOperator{\cov}{Cov}
\DeclareMathOperator{\supp}{supp}
\DeclareMathOperator{\rank}{rank}
\DeclareMathOperator*{\diag}{diag}
\theoremstyle{definition}
\newtheorem{definition}{Definition}[section]

\theoremstyle{remark}
\newtheorem{remark}{Remark}[section]
\theoremstyle{plain}

\newtheorem{lemma}{Lemma}[section]
\newtheorem{proposition}{Proposition}[section]
\newtheorem{corollary}{Corollary}[section]
\newtheorem{assumption}{Assumption}[section]

\title{Perturbation-based inference for diffusion processes: Obtaining
  effective models from multiscale data}

\author{Sebastian Krumscheid} 
\address{CSQI, Institute of Mathematics, {\'E}cole Polytechnique F{\'e}d{\'e}rale de Lausanne, 1015 Lausanne, Switzerland.}
\email{sebastian.krumscheid@epfl.ch}
\date{\today}

\begin{document}

\begin{abstract}
  We consider the inference problem for parameters in stochastic
  differential equation models from discrete time observations (e.g.\
  experimental or simulation data). Specifically, we study the case
  where one does not have access to observations of the model itself,
  but only to a perturbed version which converges weakly to the
  solution of the model. Motivated by this perturbation argument, we
  study the convergence of estimation procedures from a numerical
  analysis point of view. More precisely, we introduce appropriate
  consistency, stability, and convergence concepts and study their
  connection. It turns out that standard statistical techniques, such
  as the maximum likelihood estimator, are not convergent
  methodologies in this setting, since they fail to be stable. Due to
  this shortcoming, we introduce and analyse a novel inference
  procedure for parameters in stochastic differential equation models
  which turns out to be convergent. As such, the method is
  particularly suited for the estimation of parameters in effective
  (i.e.\ coarse-grained) models from observations of the corresponding
  multiscale process. We illustrate these theoretical findings via
  several numerical examples.
\end{abstract}

\maketitle


%
%
\section{Introduction}
\label{sec:intro}

Stochastic differential equation (SDE) models play a prominent role
when studying the temporal evolution of diverse phenomena arising in a
wide range of areas. In many applications it is desirable to fit an
SDE model to discrete time observations (e.g.\ experimental or
simulation data) of the phenomenon of interest in order to use this
model for further analysis \cite{Krumscheid2015-PhysRevE}. It is often
possible to justify postulating an SDE model with a particular
structure based on theoretical arguments or previous experience with
related systems. In that case fitting the model to the available
discrete time observations corresponds to determining an unknown
parameter vector $\theta\in\R^n$ that characterises an $d$-dimensional
SDE model such as
\begin{equation}
  dX = f(X;\theta)\,dt + g(X;\theta)\,dW\;.\label{eq:intro:sde}
\end{equation}
In abstract terms, an estimator for $\theta$ can be viewed as a
mapping from the sample space (i.e.~the space of observations) to the
parameter space $\R^n$ and it is solely derived from model
\eqref{eq:intro:sde}. For concreteness, let the observations of $X$
correspond to model \eqref{eq:intro:sde} with true parameter $\theta$
and denote by $\Lambda_\lambda(X)$ the estimated value using the
procedure $\Lambda_\lambda$ based on these observations. Here
$\lambda$ is a generic parameter which accounts for effects that
influence the estimated value, such as the number of discrete time
observations or effects due to other approximations. Of particular
interest is to verify that the parameter vector $\theta$ can be
recovered asymptotically from the observations, i.e.\ it is desirable
that $\lim_{\lambda\rightarrow 0}\Lambda_\lambda(X) = \theta$ in an
appropriate sense, with $\lambda\rightarrow 0$ denoting a generic
limit value. For instance, if $\Lambda_\lambda(X)$ denotes the
continuous time maximum likelihood estimator based on the observed
path $X$ over the time interval $[0,T]$ (we will come back to this
estimator in Section \ref{sec:setting:MLE}), then we wish to recover
the true parameter asymptotically as $T\rightarrow\infty$, so that
$\lambda = 1/T$ in this case. There exists a vast and well-established
literature concerning this property, both from theoretical and
computational aspects
\cite{PrakasaRao1999,Kutoyants2004,Iacus2008,Liptser2010}. For the
special case of estimating parameters in ordinary differential
equations, i.e.\ $g\equiv 0$ in equation \eqref{eq:intro:sde}, see
\cite{Li2005} for example.

In this work, we are interested in a slightly different scenario:
instead of having direct access to observations $X$ corresponding to
model \eqref{eq:intro:sde} with true parameter $\theta$, we only
observe a process $X^\varepsilon$ which converges weakly to $X$ in the
limit of $\varepsilon\rightarrow 0$. This situation cannot easily be
ruled out in many practical applications. One such example is the
problem of inferring effective coarse-grained models from observations
of a complex or possibly unknown system with multiple temporal and/or
length scales. These multiscale systems (both deterministic and
stochastic) emerge naturally in a range of applications, including
biology \cite{Chauviere2010}, atmosphere and ocean sciences
\cite{Majda2008}, molecular dynamics \cite{Griebel2007}, materials
science \cite{Fish2009}, and fluid and solid mechanics
\cite{Horstemeyer2010,Huerre1998}. For such a multiscale system with,
e.g., two widely separated time scales one typically only has access
to discretely sampled observations of the multiscale process
$X^\varepsilon$ which converges weakly in $C([0,T],\R^d)$ to the
solution $X$ of the corresponding coarse-grained model as
$\varepsilon\rightarrow 0$. Nevertheless one is interested in
identifying parameters in the coarse-grained model solved by $X$ using
the observation $X^\varepsilon$. In other examples one might, however,
not even be aware of the fact that one observes only a perturbed
version $X^\varepsilon$ of $X$ instead of $X$. Consequently it is
indispensable in these situations to use an estimation procedure that
is robust against this perturbation of the observation, so that one
can (asymptotically) recover the unknown parameter $\theta$ also from
$X^\varepsilon$ instead of just from $X$, in the sense that
$\lim_{\varepsilon\rightarrow 0}\lim_{\lambda\rightarrow
  0}\Lambda_\lambda(X^\varepsilon) = \theta$ in an appropriate sense.

Although this kind of robustness for estimation schemes seems
certainly desirable in many applications, it has not yet been treated
systematically in the literature. Partially related problems have been
studied in the context of parametric inference for misspecified
models; see, e.g., \cite[Ch.\ $2.6$]{Kutoyants2004} and the references
therein. In this field, one is mainly concerned with
consistency-related results of an estimation procedure
$\Lambda_\lambda$ from a statistical perspective when the observations
originate from an SDE, which is not contained in the considered class
of parametrized models such as \eqref{eq:intro:sde} (i.e.\ there does
not exist a true $\theta$). More precisely, it is of interest whether
or not the estimation procedure $\Lambda_\lambda$ (e.g.\ the maximum
likelihood estimator) still converges to a well-defined limit as
$\lambda\rightarrow 0$. It is moreover known that inferring effective
coarse-grained SDE models from temporal observations of a multiscale
system by means of estimators such as the quadratic variation of the
path estimator or maximum likelihood estimator is sometimes
impossible, since these estimators can be inconsistent (i.e.\
asymptotically biased) due to the multiscale structure of the data
\cite{Papavasiliou2009,Pavliotis2007,Pavliotis2012,Azencott2010}.  As
such, many commonly used statistical inference techniques might not be
endowed with the desirable robustness property motivated above, thus
making an accurate estimation of $\theta$ in \eqref{eq:intro:sde}
impossible, or doubtful at best. Similar consistency concerns may thus
be relevant also in any technique that relies on a stochastic
differential equation model, which had been identified form available
multiscale data, including widely applied techniques such as
stochastic filtering and stochastic control. Related work on
stochastic filtering and stochastic control for SDEs with multiple
scales can, e.g., be found in \cite{Imkeller2013,Zhang2014}.

Motivated by this potential insufficiency of statistical inference
techniques for diffusion processes, the main objective of the present
study is twofold. Firstly, we devise a numerical analysis oriented
point of view on the convergence of a general estimation
procedure. Specifically, we will introduce appropriate consistency,
stability, and convergence concepts by merging tools from mathematical
statistics and numerical analysis. This combined consistency and
stability analysis framework for inference problems is motivated by
the well-known fact in numerical analysis that consistency of a method
is not sufficient to guarantee an accurate solution to a numerical
problem \cite{Lax1956}. Secondly, we introduce a novel parametric
inference methodology that is convergent within this framework and, as
such, it is in particular robust with respect to weak perturbations,
in the sense that $\lim_{\varepsilon\rightarrow
  0}\lim_{\lambda\rightarrow 0}\Lambda_\lambda(X^\varepsilon) =
\theta$ for any $X^\varepsilon$, which converges weakly in
$C([0,T],\R^d)$ to $X$ as $\varepsilon\rightarrow 0$.  This
methodology is motivated by the recent computational studies
\cite{Krumscheid2013,Kalliadasis2015}. In fact, by subsequently
generalising and extending ideas presented in these works we obtain a
methodology which is more amenable to a rigorous convergence
analysis. The main element is to obtain an appropriate functional
relation between the unknown parameter vector $\theta$ and the
statistical properties of the model \eqref{eq:intro:sde}. From this
resulting estimating equation, we will derive an estimator of $\theta$
via the best approximation of a system of equations. Depending on the
available observation design, i.e.\ either many short trajectories are
available or only one long time series is available, we incorporate
the discretely sampled observations into this framework by replacing
theoretical conditional moments by data-driven approximations.

In the absence of input perturbations (i.e.\ when observing a solution
to \eqref{eq:intro:sde} directly), the methodology introduced here
shares some similarities with the generalised method of moments
\cite{Hansen1982}. In fact, both methods are based on deriving
parametric estimators from an appropriate estimating equation that
involves moments of a solution $X$ to the SDE \eqref{eq:intro:sde}.
In the generalised method of moments, such an equation typically
exploits ergodicity of the process $X$ and involves moment conditions
with respect to the invariant distribution. Conversely, the
methodology introduced here uses an estimating equation that accounts
for moments of (short) local transitions. Incorporating these
transitions allows for identifying all parameters $\theta\in\R^n$ in
\eqref{eq:intro:sde} at once. As a matter of fact, this is not
possible when relying only on ergodic averages because different
process may have the same invariant distribution. For example, both
diffusion processes $dX = -\alpha X\,dt+\sqrt{2\sigma}\,dW_t$ and
$dY = -\alpha\beta Y\,dt+\sqrt{2\beta\sigma}\,dW_t$ with
$\alpha,\beta,\sigma>0$ have the same invariant distribution
$\mathcal{N}(0,\sigma/\alpha)$.

The rest of this work is structured as follows. We begin, in Section
\ref{sec:setting}, by introducing a numerical analysis oriented
inference framework for diffusion processes. As an example, we study
the maximum likelihood estimator concerning its convergence properties
within this framework. In Section \ref{sec:estimator} we introduce the
novel class of estimation procedures for which we present the
convergence analysis in Section \ref{sec:conv}. To support the
theoretical findings, we investigate several data-driven
coarse-graining examples in Section \ref{sec:numerics}. Conclusions
and open questions are offered in Section \ref{sec:conclusion}.

%
%
\section{Parametric Inference Framework for Diffusion Processes}
\label{sec:setting}
Throughout this work, let $(\Omega,\mathcal{F},{(\mathcal{F}_t)}_{t\in
  [0,T]},\PP)$ be a complete, filtered probability space satisfying
the usual conditions. Furthermore, let $W = \{W(t)\colon t\in[0,T]\}$
be an $r$-dimensional Brownian motion with respect to
${(\mathcal{F}_t)}_{t\in [0,T]}$. We consider a $d$-dimensional
It{\^o} stochastic differential equation (SDE),
\begin{equation}
  dX = f(X)\,dt + g(X)\,dW_t\;,\quad X(0) = \xi\;,\label{eq:sde:generic}
\end{equation}
over a finite time interval $[0,T]$, $T>0$. The initial condition
$\xi\in\R^d$ is assumed to be independent of the $\sigma$-field
generated by $W$ and such that $\E(\norm{\xi}_2^p)<\infty$ for any
$p\ge 2$, where $\norm{\cdot}_2$ denotes the Euclidean norm in
$\R^d$. Moreover, $f\colon\R^d\rightarrow\R^d$ and
$g\colon\R^d\rightarrow\R^{d\times r}$ are assumed to be such that
\eqref{eq:sde:generic} has a unique strong solution on $[0,T]$; see
e.g.\ \cite{Karatzas1991,Oksendal2003}.

The parametric inference problem for diffusion processes, i.e.\ for
solutions of SDEs, is then the following. Let both the function $f$
and the function $g$ in \eqref{eq:sde:generic} depend on some unknown
vector-valued parameter $\theta\in\R^n$, $n\in\N$, so that
\eqref{eq:sde:generic} reads
\begin{equation}
  dX = f(X;\theta)\,dt + g(X;\theta)\,dW_t\;.\label{eq:sde:param:est:generic}
\end{equation}
We assume that \eqref{eq:sde:param:est:generic} has a unique strong
solution for any admissible parameter
$\theta\in\Theta\subseteq\R^n$. Based only on available observations
of the solution to \eqref{eq:sde:param:est:generic}, the goal then is
to accurately infer the unknown parameter $\theta$ in
\eqref{eq:sde:param:est:generic} from the observations.

An estimator for a parameter vector in SDEs is given as a mapping of
the sample space to the space of admissible parameters $\Theta$ (cf.\
\cite{Kutoyants2004,PrakasaRao1999}). Based on available observations
of the diffusion process $X$ solving \eqref{eq:sde:param:est:generic}
with parameter $\theta\in\Theta$, an estimate of $\theta$ is then
given by applying this mapping to the observations.  With slight abuse
of notation, throughout this work we will denote by $X$ both the
process solving \eqref{eq:sde:param:est:generic} and observations of
this process, whenever a distinction is not crucial.  Let
$\Lambda_\lambda(X)$ denote such an estimated value based on the
observations $X$. Here we introduce a generic, possibly vector-valued,
parameter $\lambda$ to account for the fact that the estimated value
$\Lambda_\lambda(X)$ depends on properties of the available
observations, such as the number of observations or approximations of
continuous objects (e.g.\ integrals or discretely sampled
observations). We emphasise that, although, we use only one parameter
$\lambda$ to index this family of estimators $\Lambda_\lambda$, the
generic limit $\lambda\rightarrow 0$ is merely meant as a notation for
considering the limit of all properties that influence the estimated
value, such as, for example, taking the number of observations to
infinity and the mesh size of any discretization to zero. Ultimately,
the question is whether or not the estimated value
$\Lambda_\lambda(X)$ is an accurate approximation of $\theta$. To make
this concept more precise we introduce two consistency concepts, which
express purely statistical ideas. The first one introduces the class
of feasible processes $F$, i.e.\ the class of processes for which the
estimation procedure $\Lambda_\lambda$ has a well-defined limit as
$\lambda\rightarrow 0$.
\begin{definition}[Numerical Consistency]
  \label{def:num:cons}
  Let $X$ be the solution to \eqref{eq:sde:param:est:generic}
  associated with parameter $\theta\in\Theta$ and let
  $\Lambda_\lambda$ be an estimation procedure for $\theta$. The
  procedure $\Lambda_\lambda$ is called \emph{numerically consistent}
  for class $F$, if $\lim_{\lambda\rightarrow 0}\Lambda_\lambda(Y) =:
  \Lambda (Y)$ exists in probability for any $Y\in F$. The class $F$
  is called the class of feasible processes and is such that $X\in F$.
\end{definition}

The class $F$ can be thought of as the domain of definition of the
estimation procedure, in the sense that it typically contains all
processes such that the estimated value exists in the limit as
$\lambda\rightarrow 0$. Moreover, it is natural to require that
$X\in F$, as it is not possible to estimate $\theta$ accurately using
the methodology $\Lambda_\lambda$ otherwise. The second consistency
concept given below then links the limiting value $\Lambda(X)$ to the
sought-after parameter $\theta$.
\begin{definition}[Model Consistency]
  \label{def:model:cons}
  Let $X$ be the solution to \eqref{eq:sde:param:est:generic}
  associated with parameter $\theta\in\Theta$. A numerically
  consistent estimation procedure $\Lambda_\lambda$ for $\theta$ is
  called \emph{model consistent}, if
  $\Lambda(X) \equiv \lim_{\lambda\rightarrow 0}\Lambda_\lambda(X) =
  \theta$ in probability.
\end{definition}

\begin{remark}
  \label{rem:connection:stat:consitency}
  The notion of a consistent estimation procedure commonly used in the
  mathematical statistics literature is a special case of the
  consistency concept introduced in Definition
  \ref{def:model:cons}. To see this, we assume that the estimation
  procedure $\Lambda_\lambda$ depends only on the number of
  observations, that is $1/\lambda$ denotes the number of available
  observations. Furthermore, we assume that $\Lambda_\lambda$ is
  numerically consistent for class $F = \{X\}$. Then model consistency
  of $\Lambda_\lambda$ in view of Definition \ref{def:model:cons}
  coincides with the consistency concept used in mathematical
  statistics; see, e.g., \cite{Lehmann1998,vanderVaart2000}. The
  reason for considering a more general consistency concept here is
  that we will also be concerned with additional approximation errors
  as well as perturbations to the input $X$, both of which will
  influence the convergence.
\end{remark}

As it is well-known in numerical analysis, consistency of a numerical
method is not sufficient to guarantee an accurate solution to a
numerical problem, since small perturbations in the input may result
in drastic changes in the solution. Therefore, a stability condition
is typically employed. To study the effect of ``small'' perturbations
to the input in the context of parametric inference for diffusion
processes, we consider perturbations in the following sense.
\begin{definition}[Weak perturbations]
\label{def:input:perturbation}
Let $(\Omega,\mathcal{F},\PP)$ be a probability space and let
$X^\varepsilon$, $\varepsilon>0$, and $X$ be stochastic processes
defined on that space, whose trajectories are almost surely continuous
on the time interval $[0,T]$ with values in $\R^d$. We say
$X^\varepsilon$ is a \emph{weak perturbation} of $X$, if
  \begin{equation}
    \lim_{\varepsilon\rightarrow 0}\sup_{t\in [0,T]}\abs{\E\Bigl(\varphi\bigl(X^\varepsilon(t)\bigr)\Bigr) - 
      \E\Bigl(\varphi\bigl(X(t)\bigr)\Bigr)} = 0\label{eq:def:weak:conv}
  \end{equation}
  for every $\varphi\in C_b(\R^d)$.
\end{definition}

A closely related concept is that of weak convergence of measures (see
e.g.\ \cite[Ch.\ IV.$30$]{Bauer2001_eng}), in the sense that a
sufficient condition for $X^\varepsilon$ to be a weak perturbation is
to converge weakly in $C([0,T],\R^d)$ to $X$.  Based on these weak
perturbations, we introduce a natural stability condition in context
of parametric inference for diffusion processes.
\begin{definition}[$\varepsilon$-stability]
  \label{def:eps:stab}
  Let $X$ be the solution to \eqref{eq:sde:param:est:generic}
  associated with parameter $\theta\in\Theta$. Moreover, let the
  estimation procedure $\Lambda_\lambda$ for $\theta$ be numerically
  consistent for class $F$. Then $\Lambda_\lambda$ is called
  \emph{$\varepsilon$-stable}, if
  $\lim_{\varepsilon\rightarrow 0}\Lambda(X^\varepsilon) = \Lambda(X)$
  in probability for any weak perturbation $X^\varepsilon$ of $X$,
  which is such that $X^\varepsilon\in F$.
\end{definition}
\begin{remark}
  \label{rem:stab:cont}
  The concept of $\varepsilon$-stability of an estimation procedure
  can also be viewed as a continuity property of that procedure. In
  fact, the $\varepsilon$-stability condition
  $\lim_{\varepsilon\to 0}\Lambda(X^\varepsilon) =
  \Lambda(\lim_{\varepsilon\to 0}X^\varepsilon)$, implies that the
  (asymptotic) estimation procedure
  $\Lambda(X^\varepsilon) \equiv \lim_{\lambda\to
    0}\Lambda_\lambda(X^\varepsilon)$, viewed as a function of
  $\varepsilon$, is (asymptotically) continuous in
  $\varepsilon=0$. Conversely, an $\varepsilon$-unstable estimation
  procedure is discontinuous in $\varepsilon=0$.
\end{remark}
Regardless of the consistency and stability concepts developed above,
ultimately we are interested whether or not an estimation procedure
for $\theta$ in \eqref{eq:sde:param:est:generic} yields an accurate
approximation when applied to a weak perturbation $X^\varepsilon$ of
$X$. Only when the estimated value based on weak perturbations
coincides with the true value $\theta$ asymptotically, we call a
estimation methodology convergent. The following definition makes this
intuition precise.
\begin{definition}[Convergence]
  \label{def:est:conv}
  Let $X$ be the solution to \eqref{eq:sde:param:est:generic}
  associated with parameter $\theta\in\Theta$. An estimation procedure
  $\Lambda_\lambda$ for $\theta$ is called \emph{convergent} for class
  $F$, if
  \begin{equation*}
    \lim_{\varepsilon\rightarrow 0}\lim_{\lambda\rightarrow 0}\Lambda_\lambda(X^\varepsilon) = \theta\
  \end{equation*}
  in probability for any weak perturbation $X^\varepsilon$ of $X$,
  such that $X^\varepsilon\in F$.
\end{definition}

There is a natural link between the consistency and stability concepts
introduced above, and this convergence concept.
\begin{lemma}
  Let $\Lambda_\lambda$ be an estimation procedure for the parameter
  $\theta$ in \eqref{eq:sde:param:est:generic}. If $\Lambda_\lambda$
  is model consistent and $\varepsilon$-stable for class $F$, then
  $\Lambda_\lambda$ is convergent for class $F$. Conversely, if
  $\Lambda_\lambda$ is convergent for class $F$ and model consistent,
  then $\Lambda_\lambda$ is $\varepsilon$-stable for class $F$.
\end{lemma}
\begin{proof}
  Let $X$ be the solution to \eqref{eq:sde:param:est:generic}
  associated with parameter $\theta\in\Theta$ and let $F$ be the class
  of processes for which $\Lambda_\lambda$ is numerically
  consistent. The fact that model consistency and
  $\varepsilon$-stability imply convergence then follows from the
  bound
  $\norm{\Lambda_\lambda(X^\varepsilon) - \theta}_2\le
  \norm{\Lambda_\lambda(X^\varepsilon) - \Lambda_\lambda(X)}_2 +
  \norm{\Lambda_\lambda(X) - \theta}_2$, since the right-hand side
  vanishes for any $X^\varepsilon\in F$ as $\varepsilon\to 0$ and
  $\lambda\to 0$. Similarly, model consistency and convergence imply
  that the right-hand side of
  $\norm{\Lambda_\lambda(X^\varepsilon) - \Lambda_\lambda(X)}_2\le
  \norm{\Lambda_\lambda(X^\varepsilon) - \theta}_2 +
  \norm{\Lambda_\lambda(X) - \theta}_2$ vanishes asymptotically, which
  shows the $\varepsilon$-stability.
\end{proof}%
In other words, the Lemma above states that stability is a necessary
and sufficient condition for the convergence of a consistent
estimation methodology. This relationship resembles the essence of the
Lax equivalence theorem \cite{Lax1956}, at least in the context of
linear problems.

\begin{remark} 
  \label{rem:illposed:prob}
  By casting the parametric inference problem into a numerical
  analysis framework, one notices the resemblance to inverse problems
  and to regularisation techniques. In fact, there is direct link to
  the concept of well-posed problems in the sense of Hadamard, as such
  that $\varepsilon$-stability reflects the dependency of the solution
  on perturbations of the input argument. Consequently, the parametric
  inference problem using an $\varepsilon$-unstable method would not
  be well-posed and it had to be regularised for its numerical
  treatment. Typical regularisation techniques reformulate the problem
  by incorporating additional information (e.g.\ regularity
  assumptions) or constraints to obtain a well-posed problem. We will
  briefly come back to this point in Remark \ref{rem:subsampling}.
\end{remark}

\subsection{The maximum likelihood estimator for multiscale diffusion processes}
\label{sec:setting:MLE}
In this Section we consider the maximum likelihood estimator (MLE) in
continuous time to illustrate the concepts introduced
above. Specifically, we focus on a simple one-dimensional example
borrowed from \cite{Pavliotis2007}. Consider the case where the SDE
\eqref{eq:sde:param:est:generic} is the first order Langevin equation,
given by
\begin{equation}
  dX = -A V'(X)\, dt + \sqrt{2\Sigma}\,dW_t\;,\label{eq:langevin:eff:sde}
\end{equation}
with $A,\Sigma >0$. We assume that $\Sigma$ is known so that we are
only concerned with estimating the parameter $A$ from a trajectory of
continuous time observations on the time interval $[0,T]$, $T>0$. Let
$V\colon\R\rightarrow \R$ be a confining potential with at most
polynomial growth, for which there exist $c_1,c_2>0$ such that
$-V'(x)x\le c_1 - c_2 x^2$ for every $x\in\R$ (e.g.\ $V(x) =
x^2/2$). Consequently, the solution $X$ to \eqref{eq:langevin:eff:sde}
is ergodic. Then the MLE for $A$ is given by (see
\cite{PrakasaRao1999,Kutoyants2004})
\begin{equation}
  \Lambda_T(X) := -\,\frac{\int_0^T V'\bigl(X(t)\bigr)\, dX(t)}{\int_0^T\abs{V'\bigl(X(t)\bigr)}^2\,dt}\;,\label{eq:langevin:mle}
\end{equation}
where we have indexed the class of estimators by $T$ instead of
$\lambda$, as $\lambda = 1/T$ here. Mimicking the proof of \cite[Thm.\
$3.4$]{Pavliotis2007}, one readily obtains numerical consistency of
the MLE for a class of ergodic diffusion processes.
\begin{lemma}[MLE is numerical consistent]
  \label{lem:mle:num:cons}
  Let $F$ be defined as
  \begin{align*}
    F = \Biggl\{Y\in C\bigl([0,\infty)\bigr)&\colon dY = b(Y)\,dt + \sqrt{2\gamma}\,dW_t\;,\; Y
    \text{ ergodic with meas.\ } \mu  \text{ and }
    \tfrac{\vert\int bV'\,d\mu\vert}{\int {\vert
        V'\vert}^2\,d\mu} < \infty\Biggr\}\;,
  \end{align*}
  with $V\colon\R\to\R$ being the confining potential in
  \eqref{eq:langevin:eff:sde}. Then the MLE $\Lambda_T$ in
  \eqref{eq:langevin:mle} is numerical consistent for class $F$, in
  the sense that $\lim_{T\to\infty}\Lambda_T(Y)$ is almost surely
  finite for any $Y\in F$.
\end{lemma}

Clearly, the solution $X$ to \eqref{eq:langevin:eff:sde} is in
$F$. Furthermore, model consistency of the MLE is a well-known fact in
the mathematical statistics literature; see
\cite{PrakasaRao1999,Kutoyants2004,Liptser2010} for example.
\begin{lemma}[MLE is model consistent]
  \label{lem:mle:mod:cons}
  Let $X$ be the solution to \eqref{eq:langevin:eff:sde} corresponding
  to the parameters $A,\Sigma>0$. Then the MLE $\Lambda_T$ for $A$ is
  model consistent, so that $\lim_{T\rightarrow\infty}\Lambda_T(X) = A$ in
  probability.
\end{lemma}

Despite the consistency results of Lemmas \ref{lem:mle:num:cons} and
\ref{lem:mle:mod:cons}, an accurate numerical treatment of the
parametric inference problem for the SDE model
\eqref{eq:langevin:eff:sde} via the MLE is still not guaranteed. In
fact, the MLE fails to be $\varepsilon$-stable and it is, as such, not
a convergent estimation procedure. To see this, we construct a weak
perturbation in $F$, for which the MLE is not
convergent. Specifically, consider the SDE
\begin{equation}
  dX^\varepsilon = -\alpha V'(X^\varepsilon)\,dt - \frac{1}{\varepsilon} p'(X^\varepsilon/\varepsilon)\,dt + \sqrt{2\sigma}\,dW_t\;,\label{eq:langevin:fast:sde}
\end{equation}
with $p$ being a smooth periodic function with period $L>0$ and let
$\varepsilon>0$. Let
$Z_{\pm}(\sigma) = \int_0^Le^{\pm p(y)/\sigma}\,dy$ and define
$R(\sigma) = L^2/(Z_{+}(\sigma)Z_{-}(\sigma))$. Notice that
$0<R(\sigma)<1$ in view of the Cauchy--Schwarz inequality. Then for
$\alpha$, $\sigma$ such that $\alpha R(\sigma) = A$ and
$\sigma R(\sigma) = \Sigma$ it is known that $X^\varepsilon$ solving
\eqref{eq:langevin:fast:sde} converges weakly in $C([0,T],\R)$ to $X$
in the limit as $\varepsilon\rightarrow 0$.  In other words,
$X^\varepsilon$ is a weak perturbation of $X$ solving
\eqref{eq:langevin:eff:sde} in the sense of Definition
\ref{def:input:perturbation}. Moreover, the process $X^\varepsilon$ is
ergodic for any $\varepsilon>0$ \cite[Prop.\ $5.2$]{Pavliotis2007} and
it follows that $X^\varepsilon\in F$, where $F$ is as in Lemma
\ref{lem:mle:num:cons}. Thus, the consistency results of Lemmas
\ref{lem:mle:num:cons} and \ref{lem:mle:mod:cons} imply that
\begin{equation*}
  \lim_{\varepsilon\rightarrow 0}\lim_{T\rightarrow\infty}\abs{A-\Lambda_T(X^\varepsilon)} \ge 
  \lim_{\varepsilon\rightarrow 0}\lim_{T\rightarrow\infty}\bigl\vert\abs{A-\Lambda(X^\varepsilon)} - \abs{\Lambda(X^\varepsilon)-\Lambda_T(X^\varepsilon)}\bigr\vert
  \ge \lim_{\varepsilon\rightarrow 0} \abs{\Lambda(X^\varepsilon)-\Lambda(X)}\;,
\end{equation*}
holds in probability. It follows from \cite[Thm.\
$3.4$]{Pavliotis2007} that
$\lim_{\varepsilon\rightarrow 0}
\abs{\Lambda(X^\varepsilon)-\Lambda(X)} =
A\frac{\abs{1-R(\sigma)}}{R(\sigma)}>0$, which shows that the MLE is
not $\varepsilon$-stable for the perturbation $X^\varepsilon\in
F$. Moreover, we find that
$\lim_{\varepsilon\rightarrow
  0}\lim_{T\rightarrow\infty}\abs{A-\Lambda_T(X^\varepsilon)} >0$ for
this process. That is, the process $X^\varepsilon\in F$ is a
counterexample showing that the MLE cannot be convergent for class
$F$.

Finally, it is noteworthy that not just the MLE fails to be
$\varepsilon$-stable, but that also the underlying likelihood
function, from which the MLE expression in \eqref{eq:langevin:mle}
eventually follows, is drastically affected by the weak perturbation
$X^\varepsilon$. In fact, it is known that the (asymptotic) likelihood
function itself is corrupted by a non-constant bias term (as a
function of the parameter $\theta$) when confronted with a weak
perturbation $X^\varepsilon$ \cite[Thm.\ 3.12]{Papavasiliou2009}.
\begin{remark}
  \label{rem:subsampling}
  As the MLE is not convergent, for it to become a meaningful
  inference scheme appropriate regularisation techniques have to be
  used, as we have mentioned in Remark \ref{rem:illposed:prob}
  already. Although not coined as such, the principle of data
  subsampling for parametric inference (see, e.g.,
  \cite{Pavliotis2007,Papavasiliou2009,Pavliotis2012,Azencott2010,Azencott2011,Azencott2013})
  can be viewed as such a regularisation technique as one introduces
  additional conditions on the sampling rate. In fact, subsampling the
  data at an optimal rate can make the MLE \eqref{eq:langevin:mle}
  convergent for class $F$; see \cite{Pavliotis2007}. Related work on
  parametric inference based on multiscale data combined with
  subsampling techniques can also be found in
  \cite{Zhang2005,Cotter2009,Olhede2009,Crommelin2011,Crommelin2012}
  for example, while the references \cite{Spiliopoulos2013,Gailus2017}
  contain work on the MLE for multiscale problems in the case of
  vanishing noise intensity. We emphasise, however, that the optimal
  sampling rate is typically unknown and that it can also vary for
  different parameters in the same model, thus making a subsampling
  approach often inefficient in practise.
\end{remark}

%
%
\section{A Parametric Inference Technique for Diffusion Processes}
\label{sec:estimator}

Here we introduce a procedure for the parametric inference problem of
diffusion processes which is motivated by the recent computational
results in \cite{Krumscheid2013,Kalliadasis2015}. In fact, we extend
and generalise the introduced procedure further to make it more
amenable to a theoretical treatment. Specifically, consider the
following $d$-dimensional It{\^o} SDE
\begin{equation}
  dX = f(X)\,dt + g(X)\,dW_t\;,\quad X(0) = \xi\;,\label{eq:sde}
\end{equation}
where $f\colon\R^d\rightarrow\R^d$,
$g\colon\R^d\rightarrow\R^{d\times r}$, and $W$ denotes a standard
$r$-dimensional Brownian motion. The initial condition $\xi\in\R^d$ is
assumed to be deterministic and, as before, both functions $f$ and $g$
are assumed to be such that \eqref{eq:sde} has a unique strong
solution on any finite time interval $[0,T]$, $T>0$. In what follows,
we will use $X_\xi(t)$ to denote a solution of \eqref{eq:sde} at
time $t\in[0,T]$ started in $\xi$ at time zero, i.e.\
$X_\xi(0) = \xi$. Moreover, let $\mathcal{L}$ be the generator of the
diffusion process \eqref{eq:sde}, i.e.\
\begin{equation*}
  \mathcal{L}\phi = f\cdot\nabla\phi + \frac{1}{2} G:\nabla\nabla\phi\;,
\end{equation*}
with $G:=gg^T\colon\R^d\rightarrow\in\R^{d\times d}$ and where
$A:B\equiv\tr(A^TB)$ denotes the Frobenius inner product of matrices
$A,B\in\R^{d\times d}$. Then for any $\phi\in C^2\bigl(\R^d\bigr)$,
It{\^o}'s formula implies that
\begin{equation}
  \E\Bigl(\phi\bigl(X_\xi(t)\bigr)\Bigr) - \phi(\xi)
   = \int_0^t\E\Bigl((\mathcal{L}\phi)\bigl(X_\xi(s)\bigr)\Bigr)\,ds\;,\label{eq:ito:est:phi}
\end{equation}
when additionally assuming that $\phi$, $f$, and $g$ are sufficiently
regular so that Fubini's theorem holds.

For the parametric inference problem we assume that both drift $f$ and
diffusion $G = gg^T$ depend on unknown parameters
$\theta=(\theta_1,\dots,\theta_n)^T\in\Theta = \R^n$, which we wish to
estimate from available data (i.e.\ observations). Specifically, we
consider the case where $f$ and $G$ can be expressed as a series
expansion using appropriate functions ${(f_j)}_{1\le j\le n}$ and
${(G_j)}_{1\le j \le n}$, respectively. That is, both drift function
and diffusion function depend linearly on $\theta$, so that
\begin{equation}
  f(x)\equiv f(x;\theta) := \sum_{j=1}^n\theta_jf_j(x)\quad\textrm{and}\quad
  G(x)\equiv G(x;\theta) := \sum_{j=1}^n\theta_jG_j(x)\;,\label{eq:parameterization}
\end{equation}
with $f_j\colon\R^d\rightarrow\R^d$ and
$G_j\colon\R^d\rightarrow\R^{d\times d}$ for $1\le j\le n$. Notice
that this parametric form does not imply that both the drift function
and the diffusion function have to depend on the same parameters,
because $f_j$ (or $G_j$) can vanish for suitable indices (see also
Section \ref{sec:numerics}). We also remark that the representation
\eqref{eq:parameterization} is always possible if $f$ and $G$ belong
to some finite dimensional vector space with basis functions $f_j$ and
$G_j$, respectively. For the numerical examples in Section
\ref{sec:numerics} we will typically take $f$ and $G$ to be
polynomials of some degree and use monomial basis functions. The
semiparametric representation \eqref{eq:parameterization} makes the
inference problem finite dimensional and will eventually lead to a
linear least squares problem.

Substituting the parametrization \eqref{eq:parameterization} into
\eqref{eq:ito:est:phi} and rearranging the terms, we find
\begin{equation}
  \E\Bigl(\phi\bigl(X_\xi(t)\bigr)\Bigr) - \phi(\xi) 
   = \sum_{j=1}^n\theta_j\int_0^t\E\Bigl((\mathcal{L}_j\phi)\bigl(X_\xi(s)\bigr)\Bigr)\,ds\;,
   \label{eq:ito:est:phi:param}
\end{equation}
where
$\mathcal{L}_j\phi := f_j\cdot\nabla\phi +
\frac{1}{2}G_j:\nabla\nabla\phi$. For any time $t\in [0,T]$ and any
function $\phi$ we define the local contribution functions
\begin{align*}
  b_c\colon\R^d\ni\xi\mapsto b_c(\xi) &\equiv b_c(\xi,t,\phi,X) :=  \E\Bigl(\phi\bigl(X_\xi(t)\bigr)\Bigr) - \phi(\xi) \in\R\;,\\
  a_c\colon\R^d\ni\xi\mapsto a_c(\xi) &\equiv a_c(\xi,t,\phi,X) :=
  {\biggl(\int_0^t\E\Bigl((\mathcal{L}_j\phi)\bigl(X_\xi(s)\bigr)\Bigr)\,ds\biggr)}_{1\le
    j\le n}\in\R^n\;,
\end{align*}
for the sake of notation. In fact, then equation
\eqref{eq:ito:est:phi:param} can be written as
\begin{equation}
  a_c(\xi)^T\theta = b_c(\xi)\;.\label{eq:fun:form:param:single}
\end{equation}
As equation \eqref{eq:fun:form:param:single} is under-determined for
$n>1$, we derive a well-defined estimator for $\theta$ by exploiting
the fact that equation \eqref{eq:fun:form:param:single} is valid for
any $\xi\in\R^d$. Specifically, by considering a finite sequence of
trial points ${(\xi_{i})}_{1\le i\le m}$ we find that $\theta$ solves
the linear system of equations
\begin{equation}
  A\theta = b\;,\label{eq:fun:form:param:sys}
\end{equation}
with matrix
$A := \bigl(a_c(\xi_{i})^T\bigr)_{1\le i\le m}\in\R^{m\times n}$ and
right-hand side $b:= \bigl(b_c(\xi_{i})\bigr)_{1\le i\le
  m}\in\R^{m}$. We emphasise that both the matrix $A$ and the
right-hand side $b$ depend on the considered trial points
$\Xi := {(\xi_i)}_{1\le i\le m}$, say, as well as $t$, $\phi$, and the
process $X$ solving \eqref{eq:sde}, that is $A\equiv A(X,t,\phi,\Xi)$
and $b\equiv b(X,t,\phi,\Xi)$.

In view of \eqref{eq:fun:form:param:sys}, the inference problem for
$\theta$ in a continuous setting reduces to solving a linear
system. As the matrix $A$ is typically singular, and the right-hand
side $b$ might not be in the range of $A$, we define the estimator of
$\theta$ based on $A$ and $b$ as the least squares solution of
$A\theta = b$ with minimum norm
\begin{equation}
  \hat{\theta} := \argmin_{x\in\mathcal{S}}\norm{x}_2^2\;,
  \quad\mathcal{S}:=\bigl\{x\in\R^n\colon \norm{Ax-b}_2^2 = \min\bigr\},\label{eq:param:lsp}
\end{equation}
equivalently written as $\hat{\theta} = A^{+}b$, with $A^{+}$ denoting
the pseudoinverse of $A$ \cite{Ben-Israel2003}. It is well known that
the least squares solution \eqref{eq:param:lsp} is always unique
\cite[Thm.\ $1.2.10$]{Bjorck1996}. Consequently, the estimator
$\hat{\theta}$ is well-defined. Notice that, by construction, the true
parameter $\theta$ satisfies equation \eqref{eq:fun:form:param:sys},
so that $\theta\in\mathcal{S}$.  However, $\theta \ne \hat\theta$ is
still possible, since there might be more than one element in $\R^n$
that minimises $x\mapsto\norm{Ax-b}_2^2$. This is due to the fact that
we solve the linear system in the least squares sense
\eqref{eq:param:lsp}; we will come back to this problem and its
consequences in Section \ref{sec:conv:errana}. Finally, we note that
we use $\Theta=\R^n$ throughout this work for simplicity. The case
$\Theta\subset\R^n$ results in a constrained least squares problem and
can be treated similarly; cf.\ \cite[Ch.\ $5$]{Bjorck1996}.

\subsection{Admissible functions}
Both the matrix $A$ and the right-hand side $b$ in equation
\eqref{eq:fun:form:param:sys} depend on the function $\phi$, so that
also the least squares estimator $\hat\theta$ depends on it. In the
formal derivation of \eqref{eq:param:lsp} above, we have not specified
the function $\phi$ yet, except assuming sufficient regularity. The
following definition makes the assumptions on $\phi$ concrete.
\begin{definition}
  \label{def:admissible}
  The \emph{space of admissible functions}, denoted by $V_n$, is
  defined as
  \begin{equation}
    V_n := C_b\bigl(\R^d\bigr)\cap
    \bigcap_{j=1}^n \bigl\{\varphi\in C^2\bigl(\R^d\bigr)\colon \mathcal{L}_j\varphi\in C_b\bigl(\R^d\bigr)\bigr\}\;,\label{eq:def:admissible}
  \end{equation}
  where $\mathcal{L}_j\varphi = f_j\cdot\nabla\varphi +
  \frac{1}{2}G_j:\nabla\nabla\varphi$, and the functions $f_j$ and
  $G_j$ are fixed by the considered parametrization
  \eqref{eq:parameterization}.
\end{definition}

The derivation of \eqref{eq:param:lsp} above is rigorous for any
$\phi\in V_n$, since in that case both It{\^o}'s formula and Fubini's
theorem (see e.g.\ \cite[Ch.\ III.$23$]{Bauer2001_eng}) are indeed
applicable. Moreover, the reason for considering only bounded
functions is due to the fact that this not only ensures all
expectations to be finite but, more importantly, will also yield
favourable properties of the estimation procedure when confronted with
weak perturbations. Finally, it is important to note that $V_n$ is
typically nonempty. To see this, consider for example the case that
all $f_j$ and $G_j$ are continuous functions satisfying polynomial
growth conditions, respectively. Then the function
$\exp{(-\norm{x}_2^2)}p(x)$, where $p$ is an arbitrary polynomial, is
an admissible function for example. We also remark that the set of
admissible functions $V_n$ defined in \eqref{eq:def:admissible} might
not be the largest possible class. It is, however, sufficient for our
purposes since we only need one element in $V_n$ to define the
estimator $\hat\theta$.

\subsection{Fully discretized estimation procedure}
\label{sec:estimator:fully:discrete}
In practice both the matrix $A$ and the right-hand side $b$ in the
definition of the least squares problem \eqref{eq:param:lsp} are not
readily available but can only be obtained approximately based on
available observations (i.e.\ in a data-driven fashion). Hence, using
these assembled approximations of $A$ and $b$ in \eqref{eq:param:lsp}
instead, introduces an error to the estimation
procedure. Specifically, the following different error sources are
considered here:
\begin{enumerate}
\item[(a)] Sampling errors in discretely sampled observations of a
  continuous time process. Let $\mathcal{T}_h$ be the time
  discretization of $[0,T]$, then, for any $\tau\in\mathcal{T}_h$,
  only the time discrete approximation $\bar{X}_{h\vert \xi}$
  corresponding to time step $h$ is available:
  \begin{equation*}
    \bar{X}_{h\vert \xi}(\tau) \approx X_{\xi}(\tau)\;.
  \end{equation*}
\item[(b)] Errors due to approximating time integrals by numerical
  quadrature. Here we resort to the trapezoidal rule due to its
  advantages over higher order methods for a ``rough'' integrand
  \cite{Cruz-Uribe2002}, but other quadrature rules are also
  possible. Specifically, let $Q_{n_\delta}^t$ denote the quadrature
  operator of the trapezoidal rule on $[0,t]$ with $n_\delta$ equally
  spaced ($\delta = t/n_\delta$) subdivisions, so that
  \begin{equation}
    \int_0^t \varphi(s)\,ds \approx \frac{\delta}{2} \Biggl( \varphi(0) + \varphi(t) + 2\sum_{k=1}^{n_\delta - 1}\varphi(k\delta) \Biggr) =: Q_{n_\delta}^t(\varphi) \;,\label{eq:quadrature:trapezoidal}
  \end{equation}
  for an appropriate function $\varphi\colon [0,t]\to\R$.
\item[(c)] Errors due to approximating expectations. For
  $\tau\in\mathcal{T}_h$ we use an approximation
  \begin{equation}
    \E\Bigr(\varphi\bigl(\bar{X}_{h\vert\xi}(\tau)\bigr)\Bigr)\approx \bar{u}_{h,N}(\tau,\xi;\varphi)\;,\label{eq:approx:E:std}
  \end{equation}
  for which the approximation error vanishes asymptotically in a
  probabilistic sense (e.g.\ almost surely). Here,
  $\bar{u}_{h,N}(\tau,\xi;\varphi)$ could be an appropriate ensemble
  average or time average, depending on the available observations
  (see Section \ref{sec:estimator:approx:expectation} below for
  details).
\end{enumerate}
For a fixed time $t\in[0,T]$, a sequence of trial points $\Xi$, and an
admissible function $\phi\in V_n$ the right-hand side $b$ in
\eqref{eq:param:lsp} is then approximated by
\begin{equation*}
  b_{h,N}:= \Bigl( \bar{u}_{h,N}(t,\xi_i;\phi) - \phi(\xi_i) \Bigr)_{1\le i\le m}\in\R^m\;,
\end{equation*}
while the matrix $A$ by
\begin{equation*}
  A_{\delta,h,N} := \bigl(a_{\delta,h,N}(\xi_{i})^T\bigr)_{1\le i\le
  m}\in\R^{m\times n}\;,\;\;
  a_{\delta,h,N}(\xi) := {\Bigl( Q_{n_\delta}^t\bigl( \bar{u}_{h,N}(\cdot,\xi;\mathcal{L}_j\phi)\bigr) \Bigr)}_{1\le j\le n}\in\R^n\;.
\end{equation*}
The fully discretized estimation procedure is then given by
\begin{equation*}
  \hat{\theta}_{\delta,h,N} := {(A_{\delta,h,N})}^{+}b_{h,N}\;,
\end{equation*}
accordingly. To emphasise the dependency of the estimated value
$\hat{\theta}_{\delta,h,N}$ on the used observations, we will
occasionally use
\begin{equation}
  \hat{\theta}_{\delta,h,N} = {\bigl(A_{\delta,h,N}(X)\bigr)}^{+}b_{h,N}(X)=: \Lambda_\lambda(X)\;,\label{eq:param:lsp:approx:std}
\end{equation}
with $\lambda = (\delta,h,N)$, corresponding to the notation
introduced in Section \ref{sec:setting}.

\subsection{Approximating expectations from observations}
\label{sec:estimator:approx:expectation}
An important task when using the described estimation procedure for
discrete time observations is to approximate expectations from
available observations. More precisely, let $X_\xi(t)$ denote a
generic diffusion process at time $t\in[0,T]$ started at $\xi$ and
recall that $\mathcal{T}_h$ denotes a time discretization of
$[0,T]$. Furthermore, let $\bar{X}_{h\vert\xi}(\tau)$,
$\tau\in\mathcal{T}_h$, denote a time discrete approximation of
$X_\xi$. To obtain the estimated value
\eqref{eq:param:lsp:approx:std}, expectations of the form
$\E\bigr(\varphi\bigl(\bar{X}_{h\vert\xi}(\tau)\bigr)\bigr)$ for
$\varphi\in C_b(\R^d)$ need to be approximated. The choice of the
approximation depends on the design of the available observations. In
the following we consider two different observation designs: firstly
we discuss the situation when an ensemble of short trajectories is
available, and secondly the case when only one long trajectory of
observations (i.e.\ a time series) is available. We will exemplify an
approximation of the expectation in each case.

\subsubsection{Ensemble of short trajectories}
Let us first consider the case where an ensemble of independent and
identically distributed (i.i.d.) observations is available. That is,
for $h>0$ and trial point $\xi\in\R^d$ we have access to
$\bar{X}_{h\vert\xi}^{(1)}(\tau),\bar{X}_{h\vert\xi}^{(2)}(\tau),\dots$,
where $\tau\in\mathcal{T}_h$. A natural approximation of
$\E\bigr(\varphi\bigl(\bar{X}_{h\vert\xi}(\tau)\bigr)\bigr)$ with
$\varphi\in C_b(\R^d)$ is then given via an ensemble average:
\begin{equation}
  \bar{u}_{h,N}(\tau,\xi;\varphi) := \frac{1}{N}\sum_{k=1}^N\varphi\Bigl(\bar{X}_{h\vert\xi}^{(k)}(\tau)\Bigr)\;.
  \label{eq:approx:exp:ensemble}
\end{equation}
In view of the strong law of large numbers, we have the following
convergence result.
\begin{proposition}
  Let $h>0$, $\tau\in\mathcal{T}_h$, and $\xi\in\R^d$. Moreover, let
  the sequence ${\Bigl(\bar{X}_{h\vert\xi}^{(k)}(\tau)\Bigr)}_{k\ge
    1}$ be i.i.d.\ and let $\varphi\in C_b(\R^d)$. For the
  approximation \eqref{eq:approx:exp:ensemble} it then holds that
  \begin{equation*}
    \bar{u}_{h,N}(\tau,\xi;\varphi) \rightarrow \E\Bigr(\varphi\bigl(\bar{X}_{h\vert\xi}(\tau)\bigr)\Bigr)\quad\text{a.s.}\;,
  \end{equation*}
  as $N\rightarrow\infty$.
\end{proposition}

This observation design is common for many computer-based simulations
and experiments, such as, e.g., computational statistical physics, but
also some real word experiments can be cast into this framework.

\subsubsection{One long trajectory}
An observational design more prevalent in real world experiments is
when only one long trajectory of discrete time observations (i.e.\ a
time series) is available. That is, we have access to
$\bar{X}_{h}(t_1),\bar{X}_{h}(t_2),\dots$, with $0\le t_1<t_2<\dots$,
and $t_k\in\mathcal{T}_h$, $h>0$. Here we dropped the subscript for
the initial condition of the observations, since there is only one
initial condition which we cannot influence. Instead, we will obtain
an approximation of
$\E\bigr(\varphi\bigl(\bar{X}_{h\vert\xi}(\tau)\bigr)\bigr)$ by
searching the trajectory for the value of the trial point
$\xi\in\R^d$. Due to mutual dependencies between the observations in
this setting and the fact that we have to search the time series for
the value of $\xi$, we cannot expect to obtain an accurate
approximation with as little assumptions on the time discrete process
as in the ensemble case above. One technique that is nonetheless
applicable are so-called local polynomial kernel regression estimators
\cite{Fan2003,Tsybakov2009}. In the simplest case this yields the
approximation
\begin{equation}
  \bar{u}_{h,N}(\tau,\xi;\varphi) := \frac{\sum_{k=1}^N\varphi\bigl(\bar{X}_{h\vert\xi}(t_k+\tau)\bigr)K\Bigl(\frac{\bar{X}_{h\vert\xi}(t_k)-\xi}{\kappa_N}\Bigr)}{\sum_{k=1}^NK\Bigr(\frac{\bar{X}_{h\vert\xi}(t_k)-\xi}{\kappa_N}\Bigr)}\;,
 \label{eq:approx:exp:mixing}
\end{equation}
which is also known as the Nadaraya--Watson estimator
\cite{Nadaraya1964,Watson1964}. Therein $K\colon\R^d\rightarrow\R$ is
an appropriately chosen kernel and $\kappa_N>0$ denotes the bandwidth
which depends on the length $N$ of the available time
series. Throughout this work we select the Gaussian kernel
$K(x) := (2\pi)^{-d/2}\exp{(-\norm{x}_2^2/2)}$ in
\eqref{eq:approx:exp:mixing} for convenience, but we remark that other
choices are also possible.
\begin{remark}
  When defining
  $w_{N_\tau,k}(\xi) :=
  K\bigl((\bar{X}_{h\vert\xi}(t_k)-\xi)/\kappa_N\bigr)/
  \sum_{k=1}^NK\bigr((\bar{X}_{h\vert\xi}(t_k)-\xi)/\kappa_N)\bigr)$,
  one can rewrite the right-hand side in \eqref{eq:approx:exp:mixing},
  as
  $\sum_{k=1}^{N_\tau}w_{N_\tau,i}(\xi)\varphi\bigl(\bar{X}_{h\vert\xi}(t_k+\tau)\bigr)$.
  The regression estimator is thus given as a weighted average with
  non-identical weights $w_{N_\tau,k}(\xi)$. We also note that if the
  trial point $\xi$ is such that denominator in
  \eqref{eq:approx:exp:mixing} is zero, then we simply set
  $w_{N_\tau,k}(\xi) = 1/N$ for well-posedness instead.
\end{remark}

For the Gaussian kernel and under suitable conditions on the degree of
dependency of the observations, we have the following convergence
result \cite[Thm.\ $3.2$]{Bosq1998}.
\begin{proposition}
  \label{prop:conv:expect:mixing}
  Let ${\bigl(\bar{X}_{h}(t_k)\bigr)}_{k\ge 1}$ be a strictly
  stationary (discrete time) Markov process with density $p\in
  C_b^2(\R^d)$ such that $\norm{\partial_{x_i}\partial_{x_j}p}_\infty
  \le L<\infty$, for any $1\le i,j\le d$. Furthermore, let
  ${\bigl(\bar{X}_{h}(t_k)\bigr)}_{k\ge 1}$ be geometrically
  $\alpha$-mixing in the sense that
  \begin{equation*}
   \mathop{\sup_{B\in\sigma(\bar{X}_{h}(t_1))}}_{C\in\sigma(\bar{X}_{h}(t_{1+k}))}{\abs{\PP(B\cap C)-\PP(B)\PP(C)}}\le c\rho^k\;,
  \end{equation*}
  for some $\rho\in[0,1[$ and $c>0$. Let $\varphi\in C_b(\R^d)$,
  $\tau\in\mathcal{T}_h$, and $\xi\in\supp{(p)}$. If
  $\kappa_N\rightarrow 0$ at a rate such that
  $\kappa_N^dN/{\ln{(N)}}^{(2+1/\nu)}\rightarrow\infty$ as
  $N\rightarrow\infty$ for some $0<\nu<\infty$, then the approximation
  \eqref{eq:approx:exp:mixing} satisfies
\begin{equation*}
  \bar{u}_{h,N}(\tau,\xi;\varphi) \rightarrow \E\Bigr(\varphi\bigl(\bar{X}_{h\vert\xi}(\tau)\bigr)\Bigr)\quad\text{a.s.}\;,
\end{equation*}
as $N\rightarrow\infty$.
\end{proposition}
\begin{remark}
  In Proposition \ref{prop:conv:expect:mixing}, the rather technical
  $\alpha$-mixing condition on the degree of dependency of the
  observations ensures that various covariance terms can be controlled
  \cite[Ch.\ $7.2$]{Ethier1986}. Specifically, it implies that
  $\cov{\bigl(\bar{X}_h(t_1),\bar{X}_h(t_{1+k})\bigr)} \le C \rho^k$,
  for some finite $C\equiv C_h>0$ and $\rho\in[0,1[$. Related
  conditions on the covariance structure as a function of the lag $k$
  have also been used in other works on parametric inference for
  diffusion processes; see e.g.\ \cite{Azencott2011}.
\end{remark}

%
%
\section{Error Analysis for the Estimation Procedure}
\label{sec:conv}
We now analyse the estimation procedure introduced in Section
\ref{sec:estimator} concerning its convergence properties.

\subsection{Setting and Assumptions}
\label{sec:conv:setup}
Let $X$ denote the solution to the diffusion process \eqref{eq:sde} on
the time interval $[0,T]$ corresponding to the parameter
$\theta\in\Theta = \R^n$ in parametrization
\eqref{eq:parameterization}. For a fixed time $t\in [0,T]$, a sequence
of trial points $\Xi$, and an admissible function $\phi\in V_n$,
recall that $\hat{\theta}_{\delta,h,N}$ denotes the estimated value
for $\theta$ based on $X$; see \eqref{eq:param:lsp:approx:std}. That
is, in terms of the notation introduced in Section \ref{sec:setting}
we have $ \hat{\theta}_{\delta,h,N} = \Lambda_\lambda(X)$, with
$\lambda = (\delta,h,N)$. Moreover, let $X^\varepsilon$ be a weak
perturbation of $X$ and denote by
$\hat{\theta}_{\delta,h,N}^\varepsilon$ the estimated value
\eqref{eq:param:lsp:approx:std}, which is based on the observation
$X^\varepsilon$ instead of $X$:
\begin{equation}
  \hat{\theta}_{\delta,h,N}^\varepsilon := \Lambda_\lambda(X^\varepsilon)\;.\label{eq:param:lsp:approx:weak}
\end{equation}
As discussed in Section \ref{sec:estimator:fully:discrete}, the
estimation procedure is subject to different error sources. In the
following we impose assumptions to characterise these error
contributions. We begin by characterising both the accuracy of the
available discretely sampled observations and the time discretization
itself.
\begin{assumption}[Time discrete observations]
  \label{assumption:discretization:weak:convergence}
  For any $t\in[0,T]$, let $\mathcal{T}_h$ be an equidistant time
  discretization of $[0,t]$, in the sense that
  $\mathcal{T}_h = \{0,h,2h,\dots ,n_th\}$, for $h>0$ and $n_t\in\N$
  such that $t = n_t h$. The time discrete approximation
  $\bar{X}_{h\vert\xi}^\varepsilon$ corresponding to a time step $h$
  \emph{converges weakly} to $X_\xi^\varepsilon$ at time
  $\tau\in\mathcal{T}_h$ as $h\rightarrow 0$, in the sense that for
  any $\varphi\in C_P^{2+\beta}(\R^d)$, $\beta>0$ arbitrary, and any
  $\xi\in\Xi$ we have that
  \begin{equation*}
    \lim_{h\rightarrow 0} \abs{\E\Bigl(\varphi\bigl(X_\xi^\varepsilon(\tau)\bigr)\Bigr) - 
      \E\Bigl(\varphi\bigl(\bar{X}_{h\vert\xi}^\varepsilon(\tau)\bigr)\Bigr) } = 0\;.
  \end{equation*}
  Here, $C_P^k(\R^d)$ denotes the subspace of $C^k(\R^d)$, such that
  the functions, together with all their partial derivatives of orders
  smaller or equal to $k$, have at most polynomial growth.
\end{assumption}

The weak convergence assumption for the time discrete approximation is
standard and well-understood for a large class of SDEs; see
\cite{Kloeden1992}. Essentially, Assumption
\ref{assumption:discretization:weak:convergence} ensures that the
discrete time observations provide a certain accuracy. The error
contribution due to approximation of expectations, which is also
standard, is characterised next.
\begin{assumption}[Approximation of expectation]
  \label{assumption:approx:expectation}
  Let $\tau\in\mathcal{T}_h$, $h>0$, and $\xi\in\Xi$. For any
  $\varphi\in C_b(\R^d)$, the approximation
  $\bar{u}_{h,N}^\varepsilon(\tau,\xi;\varphi)$ converges almost
  surely to $\bar{u}_h^\varepsilon(\tau,\xi;\varphi) :=
  \E\bigl(\varphi\bigl(\bar{X}_{h\vert\xi}^\varepsilon(\tau)\bigr)\bigr)$
  as $N\rightarrow\infty$.
\end{assumption}

Notice that both ensemble and single trajectory based averages
$\bar{u}_{h,N}^\varepsilon$ are covered by Assumption
\ref{assumption:approx:expectation} (see Section
\ref{sec:estimator:approx:expectation} for details). Finally, we
impose a time regularity condition on the expectations, so that the
convergence of the trapezoidal rule is guaranteed.
\begin{assumption}[Approximation of time integral]
  \label{assumption:approx:integral}
  For any $\varphi\in C_b(\R^d)$, the function $t\mapsto
  \E\bigl(\varphi\bigl(X_\xi^\varepsilon(t)\bigr)\bigr) \equiv
  u^\varepsilon(t,\xi;\varphi)$ is such that
  $Q_{n_\delta}^t\bigl(u^\varepsilon(\cdot,\xi;\varphi)\bigr)$
  converges to $\int_0^t u^\varepsilon(s,\xi;\varphi)\,ds$ as
  $n_\delta\rightarrow\infty$ (or equivalently as $\delta\rightarrow
  0$, recalling that $t=\delta n_\delta$) for any fixed $t\in[0,T]$
  and any $\xi\in\Xi$.
\end{assumption}
\begin{remark}
  A sufficient condition for the convergence of the trapezoidal rule
  $Q_{n_\delta}^t$ is for the function
  $t\mapsto \E\bigl(\varphi\bigl(X_\xi^\varepsilon(t)\bigr)\bigr)$ to
  be at H{\"o}lder continuous with exponent $\alpha >0$ on $[0,t]$;
  cf.\ \cite{Cruz-Uribe2002}.
\end{remark}

In view of the introduced notations above, and omitting the dependency
on $X^\varepsilon$, the fully discretized estimator based on perturbed
input data, i.e.\ \eqref{eq:param:lsp:approx:weak}, can then
explicitly written as
\begin{equation}
  \hat{\theta}_{h,\delta,N}^\varepsilon = {(A_{h,\delta,N}^\varepsilon)}^{+}b_{h,N}^\varepsilon\;.\label{eq:param:lsp:approx:weak:explicit}
\end{equation}
In fact, therein the data-driven approximation of the right-hand side
$b$ is given by
\begin{equation*}
  b_{h,N}^\varepsilon := \Bigl( \bar{u}_{h,N}^\varepsilon(t,\xi_i;\phi) - \phi(\xi_i) \Bigr)_{1\le i\le m}\in\R^m\;,
\end{equation*}
while the data-driven approximation of the matrix $A$ by
\begin{equation*}
  A_{\delta,h,N}^\varepsilon := \bigl(a_{\delta,h,N}^\varepsilon(\xi_{i})^T\bigr)_{1\le i\le
    m}\in\R^{m\times n}\;,\;
  a_{\delta,h,N}^\varepsilon(\xi) := \Bigl( Q_{n_\delta}^t\bigl( \bar{u}_{h,N}^\varepsilon(\cdot,\xi;\mathcal{L}_j\phi)\bigr) \Bigr)_{1\le j\le n}\in\R^n\;,
\end{equation*}
accordingly.

\subsection{Convergence property}
\label{sec:conv:errana}
In view of Definition \ref{def:est:conv} the key property of the
estimation procedure for a numerically feasible result is that the
error $\norm{\theta-\hat{\theta}_{\delta,h,N}^\varepsilon}_2$ vanishes
asymptotically. Upon recalling that $\theta\in\Theta$ denotes the true
parameter in \eqref{eq:sde}, while
$\hat\theta_{\delta,h,N}^\varepsilon$ is the estimated value based on
$X^\varepsilon$ (i.e.\ given by \eqref{eq:param:lsp:approx:weak}), one
can divide the error into two parts
\begin{equation}
  \norm{\theta - \hat\theta_{\delta,h,N}^\varepsilon}_2 \le \norm{\theta - \hat\theta}_2 + 
  \norm{\hat\theta- \hat\theta_{\delta,h,N}^\varepsilon}_2\;,\label{eq:decomp:error}
\end{equation} 
where $\hat\theta$ solves \eqref{eq:param:lsp}. The first part
accounts for the error introduced by solving
\eqref{eq:fun:form:param:sys} in the least-squares sense which is not
affected by any other error sources. Hence, it vanishes if the
estimation procedure is model consistent. The second part in
\eqref{eq:decomp:error} measures the effect of the different error
contributions as well as the influence of using a weak perturbation
$X^\varepsilon$ of $X$ as input. Instead of decomposing the second
term further into one term reflecting the $\varepsilon$-stability and
one term characterising the numerical consistency, we will study the
second term in \eqref{eq:decomp:error} directly and address the
$\varepsilon$-stability and consistency concepts in Corollary
\ref{coroll:asymp:unbiased} afterwards.

For notational convenience and to facilitate the presentation of the
proofs that follow, we introduce
\begin{equation*}
  u(t,\xi;\varphi) := \E\Bigl(\varphi\bigl(X_\xi(t)\bigr)\Bigr)\;,\quad
  u^\varepsilon(t,\xi;\varphi) := \E\Bigl(\varphi\bigl(X_\xi^\varepsilon(t)\bigr)\Bigr)\;,
\end{equation*}
and, for any discretization time $\tau\in\mathcal{T}_h$,
\begin{equation*}
  \bar{u}_h^\varepsilon(\tau,\xi;\varphi) := \E\Bigl(\varphi\bigl(\bar{X}_{h\vert\xi}^\varepsilon(\tau)\bigr)\Bigr)\;.
\end{equation*}
Moreover, we recall that $\Xi={\bigl(\xi_i\bigr)}_{1\le i\le m}$
denotes the collection of considered trial points and $V_n$ is the
space of admissible functions introduced in
Definition~\ref{def:admissible}. Now we are in the position to state
the main results concerning convergence of the estimator introduced in
Section \ref{sec:estimator}.
\begin{proposition}
  \label{prop:asymp:unbiased}
  Let $X$ be the solution to \eqref{eq:sde} corresponding to the true
  parameter $\theta\in\Theta=\R^n$ in
  \eqref{eq:parameterization}. Moreover, let $\Xi$ and $\phi\in
  V_n\cap C_P^{2+\beta}(\R^d)$, for some $\beta>0$, be such that
  $\rank{(A)} = \min(m,n)$. Then, for any $t\in[0,T]$
  \begin{equation}
    \lim_{\varepsilon\rightarrow 0} \lim_{\delta\rightarrow 0} \lim_{h\rightarrow 0} \lim_{N\rightarrow\infty} 
    \norm{\hat\theta_{\delta,h,N}^\varepsilon  - \hat\theta}_2 = 0\;,\quad\textrm{a.s.}\label{eq:estimator:limit:stable}
  \end{equation}
  for any weak perturbation $X^\varepsilon$ of $X$, provided
  $X^\varepsilon$ is such that Assumptions
  \ref{assumption:discretization:weak:convergence},
  \ref{assumption:approx:expectation}, and
  \ref{assumption:approx:integral} hold for sufficiently small
  $\varepsilon>0$.

  If, moreover, $\Xi$ and $\phi\in V_n\cap C_P^{2+\beta}(\R^d)$ are
  such that $\rank{(A)} = n$, then the estimation procedure is
  convergent:
  \begin{equation}
    \lim_{\varepsilon\rightarrow 0} \lim_{\delta\rightarrow 0} \lim_{h\rightarrow 0} \lim_{N\rightarrow\infty} 
    \norm{\hat\theta_{\delta,h,N}^\varepsilon  - \theta}_2 = 0\;,\quad\textrm{a.s.}\label{eq:estimator:limit:unbiased}
  \end{equation}
\end{proposition}
\begin{proof}
  Let $X^\varepsilon$ be a weak perturbation of $X$ satisfying
  Assumptions
  \ref{assumption:discretization:weak:convergence}--\ref{assumption:approx:integral}
  for sufficiently small $\varepsilon>0$.  The difference between $b$ in
  \eqref{eq:param:lsp} and $b_{h,N}^\varepsilon$ in
  \eqref{eq:param:lsp:approx:weak:explicit} can be estimated via
  \begin{equation}
    \begin{split}
      \norm{b - b_{h,N}^\varepsilon}_2 
      &\le \sqrt{m}\max_{1\le i\le m}
      \Bigl( \abs{u(t,\xi_i;\phi) - u^\varepsilon(t,\xi_i;\phi)}\\
      &\qquad\qquad + \abs{u^\varepsilon(t,\xi_i;\phi) -
        \bar{u}_h^\varepsilon(t,\xi_i;\phi)} +
      \abs{\bar{u}_{h}^\varepsilon(t,\xi_i;\phi) -
        \bar{u}_{h,N}^\varepsilon(t,\xi_i;\phi)} \Bigr)\;.
    \end{split}
    \label{eq:perturbation:rhs:estimate}
  \end{equation}
  Since $\phi\in V_n\cap C_P^{2+\beta}(\R^d)$, the third term in
  \eqref{eq:perturbation:rhs:estimate} vanishes a.s.\ in the limit as
  $N\rightarrow\infty$ by Assumption
  \ref{assumption:approx:expectation}. Furthermore, the second term
  vanishes as $h\rightarrow 0$ in view of Assumption
  \ref{assumption:discretization:weak:convergence}, and the first term
  disappears as $\varepsilon\rightarrow 0$ in view of
  \eqref{eq:def:weak:conv}, since $\phi\in V_n$. Consequently, we find
  that
  \begin{equation}
    \lim_{\varepsilon\rightarrow 0} \lim_{h\rightarrow 0} \lim_{N\rightarrow\infty} \norm{b - b_{h,N}^\varepsilon}_2 = 0\;,\quad\textrm{a.s.}\label{eq:perturbation:rhs:limit}
  \end{equation}
  Next, we estimate the difference of matrix $A$ in
  \eqref{eq:param:lsp} and matrix $A_{\delta,h,N}^\varepsilon$ in
  \eqref{eq:param:lsp:approx:weak:explicit} via
  \begin{equation}
    \begin{split}
      \norm{A - A_{\delta,h,N}^\varepsilon}_2
      &\le \sqrt{nm}\max_{\substack{1\le i\le m \\ 1\le j\le n}} \Biggl(
      \abs{\int_0^t u(s,\xi_i;\mathcal{L}_j\phi)\, ds -\int_0^t u^\varepsilon(s,\xi_i;\mathcal{L}_j\phi)\, ds}\\
      &\qquad\qquad + \abs{\int_0^t u^\varepsilon(s,\xi_i;\mathcal{L}_j\phi)\, ds - Q_{n_\delta}^t\bigl( u^\varepsilon(\cdot,\xi_i;\mathcal{L}_j\phi)\bigr)}\\
      &\qquad\qquad
      + \abs{Q_{n_\delta}^t\bigl( u^\varepsilon(\cdot,\xi_i;\mathcal{L}_j\phi) - \bar{u}_{h}^\varepsilon(\cdot,\xi_i;\mathcal{L}_j\phi)\bigr)}\\
      &\qquad\qquad
      + \abs{Q_{n_\delta}^t\bigl( \bar{u}_{h}^\varepsilon(\cdot,\xi_i;\mathcal{L}_j\phi) - \bar{u}_{h,N}^\varepsilon(\cdot,\xi_i;\mathcal{L}_j\phi)\bigr)}\Biggr)\;.
    \end{split}
    \label{eq:perturbation:matrix:estimate}
  \end{equation}
  Recall that
  $\mathcal{L}_j\phi = f_j\cdot\nabla\phi +
  \frac{1}{2}G_j:\nabla\nabla\phi$ and that $Q_{n_\delta}^t$ denotes
  the quadrature operator of the trapezoidal rule on $[0,t]$ with
  $n_\delta$ equally spaced subdivisions, see
  \eqref{eq:quadrature:trapezoidal}. By the same argument as above, we
  find that the fourth term on the right-hand side of
  \eqref{eq:perturbation:matrix:estimate} vanishes a.s.\ as
  $N\rightarrow\infty$ by Assumption
  \ref{assumption:approx:expectation} and the third term in
  \eqref{eq:perturbation:matrix:estimate} does so in the limit as
  $h\rightarrow 0$ by Assumption
  \ref{assumption:discretization:weak:convergence}. The second term
  disappears in the limit as $\delta\rightarrow 0$ by Assumption
  \ref{assumption:approx:integral}, while the first term vanishes as
  $\varepsilon\rightarrow 0$ in view of
  \eqref{eq:def:weak:conv}. Thus, here we find
  \begin{equation}
    \lim_{\varepsilon\rightarrow 0} \lim_{\delta\rightarrow 0} \lim_{h\rightarrow 0} \lim_{N\rightarrow\infty} \norm{A - A_{\delta,h,N}^\varepsilon}_2 = 0\;,\quad\textrm{a.s.}\label{eq:perturbation:matrix:limit}
  \end{equation}
  Therefore we have that $\norm{A -
    A_{\delta,h,N}^\varepsilon}_2\norm{A^{+}}_2 < 1$ a.s.\ for
  sufficiently small $N^{-1}$, $h$, $\delta$, and $\varepsilon$. In
  view of the rank hypothesis $\rank{(A)} = \min(m,n)$ it thus follows
  from \cite[Thm.\ $1.4.2$ \& $1.4.4$]{Bjorck1996} that
  \begin{equation*}
    \norm{\hat\theta - \hat\theta_{\delta,h,N}^\varepsilon}_2 
    \le \frac{\norm{A^{+}}_2}{1-\norm{A^{+}}_2\norm{A_{\delta,h,N}^\varepsilon-A}_2}\Bigl(
    \sqrt{2} \norm{A^{+}}_2\norm{b}_2\norm{A_{\delta,h,N}^\varepsilon-A}_2 + \norm{b_{h,N}^\varepsilon-b}_2\Bigr)\;,
  \end{equation*}
  holds a.s.\ for sufficiently small $N^{-1}$, $h$, $\delta$, and
  $\varepsilon$. This bound, together with
  \eqref{eq:perturbation:rhs:limit} and
  \eqref{eq:perturbation:matrix:limit}, eventually implies the claim
  \eqref{eq:estimator:limit:stable}.
  
  For $\rank{(A)} = n$, it is well-known that the set $\mathcal{S}$ in
  \eqref{eq:param:lsp}, i.e.\ the set of all least squares solutions,
  contains only one element \cite[Thm.\ $1.1.3$]{Bjorck1996}. By
  construction $\theta\in\mathcal{S}$, so that
  $\theta=\hat\theta$. Therefore \eqref{eq:decomp:error} and
  \eqref{eq:estimator:limit:stable} imply the claim
  \eqref{eq:estimator:limit:unbiased}.
\end{proof}
\begin{remark}
  The rank condition $\rank(A) = n$ in the previous result ensures the
  model consistency of the estimation procedure. Specifically, the
  rank condition makes the link to the feasibility of parametrization
  \eqref{eq:parameterization}, in the sense that $\rank(A) = n$ is
  only possible, if the parametrization \eqref{eq:parameterization}
  for $X$ solving \eqref{eq:sde} is reasonable and unique. From a more
  technical viewpoint, the rank condition is crucial for the
  sensitivity of the least squares problem and is thus inherent to any
  methodology relying on a least squares approach. In fact, either
  rank hypothesis (i.e.\ $\rank(A) = \min\{m,n\}$ or $\rank(A) = n$)
  ensures that the least squares approach itself is stable.
\end{remark}

Based on the convergence properties of the estimation procedure
described in Proposition \ref{prop:asymp:unbiased}, it is also
possible to identify the stability and consistency concepts introduced
in Section \ref{sec:setting}. Recall that, in view of the notation
introduced in that Section, we identify $\lambda=(\delta,h,N)$ here
and understand $\lim_{\lambda\rightarrow 0}$ as
$\lim_{\delta\rightarrow 0} \lim_{h\rightarrow 0}
\lim_{N\rightarrow\infty}$. Moreover, the class $F$ of feasible
processes is characterised by processes that satisfy Assumptions
\ref{assumption:discretization:weak:convergence}--\ref{assumption:approx:integral}.
\begin{corollary}
  \label{coroll:asymp:unbiased}
  Let $X$ be the solution to \eqref{eq:sde} corresponding to the true
  parameter $\theta\in\Theta$ in \eqref{eq:parameterization}.
  Moreover, let $\Xi$ and $\phi\in V_n\cap C_P^{2+\beta}(\R^d)$, for
  some $\beta>0$, be such that $\rank{(A)} = \min(m,n)$. Then, for any
  $t\in[0,T]$, it holds that
  \begin{enumerate}
  \item[$(i)$] the estimation procedure is numerically consistent, and
  \item[$(ii)$] the estimation procedure is $\varepsilon$-stable
  \end{enumerate}
  for any weak perturbation $X^\varepsilon$ of $X$, provided
  $X^\varepsilon$ is such that Assumptions
  \ref{assumption:discretization:weak:convergence},
  \ref{assumption:approx:expectation}, and
  \ref{assumption:approx:integral} hold for sufficiently small
  $\varepsilon>0$.
\end{corollary}
\begin{proof}
  Claim $(i)$ follows from the same arguments as the ones used in the
  proof of Proposition \ref{prop:asymp:unbiased}. Under the hypotheses
  of the Corollary, result \eqref{eq:estimator:limit:stable} holds and
  claim $(ii)$ follows from $(i)$ in view of the bound
  ${\lVert\hat\theta_{\delta,h,N}^\varepsilon -\hat\theta\rVert}_2 \ge
  \bigl\vert {\lVert \hat\theta_{\delta,h,N}^\varepsilon -
  \hat\theta^\varepsilon\rVert}_2 - {\lVert
  \hat\theta^\varepsilon-\hat\theta \rVert}_2 \bigr\vert$.
\end{proof}

\begin{remark}
  From Corollary \ref{coroll:asymp:unbiased}, and in view of Remark
  \ref{rem:connection:stat:consitency}, it follows that the estimation
  procedure introduced in Section \ref{sec:estimator} is also
  consistent in the sense used in the mathematical statistics
  literature, provided that the rank condition $\rank(A) = n$
  holds. We iterate that this condition is common to all statistical
  methods relying on a least squares approach.
\end{remark}

\subsection{Convergence rates}
\label{conv:errana:rates}
From a practical point of view it is also of interest to quantify the
rate of convergence. To this end, we strengthen Assumptions
\ref{assumption:discretization:weak:convergence}--\ref{assumption:approx:integral}
by quantifying these convergence rates for the approximations
accordingly. We begin by characterising the quality of the discrete
time observations.
\begin{assumption}
  \label{assumption:discretization:weak:convergence:order}
  Let $\mathcal{T}_h = \{0,h,2h,\dots ,n_th\}$, for $h>0$, and
  $n_t\in\N$ such that $t = n_t h$. The time discrete approximation
  $\bar{X}_{h\vert\xi}^\varepsilon$ corresponding to a time step $h$
  \emph{converges weakly} with order $\beta>0$ as $h\rightarrow 0$ to
  $X_\xi^\varepsilon$ at time $\tau\in\mathcal{T}_h$, in the sense
  that
  \begin{equation}
     \abs{\E\Bigl(\varphi\bigl(X_\xi^\varepsilon(\tau)\bigr)\Bigr) - 
      \E\Bigl(\varphi\bigl(\bar{X}_{h\vert\xi}^\varepsilon(\tau)\bigr)\Bigr) } \le Ch^\beta\;,
    \label{eq:discr:weak:conv:order}
  \end{equation}
  for any $\varphi\in C_P^{2(\beta+1)}(\R^d)$ and any
  $\xi\in\Xi$. Therein $C$ is independent of $h$, for $h$ sufficiently
  small.
\end{assumption}
\begin{remark}
  Note that the analysis in this Section can be readily extended to
  non-equidistant time discretization, and the choice of an
  equidistant one is merely made for convenience. What is important,
  however, is that the time discretization is nonrandom so that a
  uniform weak convergence on the discrete interval $\mathcal{T}_h$
  follows from \eqref{eq:discr:weak:conv:order} (see \cite[p.\
  $475$]{Kloeden1992}):
  \begin{equation*}
    \max_{\tau\in\mathcal{T}_h}\abs{\E\Bigl(\varphi\bigl(X_\xi^\varepsilon(\tau)\bigr)\Bigr) - 
      \E\Bigl(\varphi\bigl(\bar{X}_{h\vert\xi}^\varepsilon(\tau)\bigr)\Bigr) } \le Ch^\beta\;.
  \end{equation*}
  Furthermore, it follows that an appropriately constructed
  continuous-time extension based on the discrete time approximations
  $\bar{X}_{h\vert\xi}^\varepsilon$ converges weakly with order
  $\beta$ on the whole interval $[0,t]$, $t\in[0,T]$.
\end{remark}
\begin{remark}
  \label{rem:eps:err:const}
  It is noteworthy that the error constant $C$ in
  \eqref{eq:discr:weak:conv:order} may depend on $\varepsilon$. This
  is possible, for example, when the discrete time observations of
  $\bar{X}^\varepsilon$ are being generated via a computer experiment
  based on discretizing an SDE with multiple time scales. However, in
  that case there exist specialised methods to remove this dependency,
  such as the heterogeneous multiscale method
  \cite{Vanden-Eijnden2003,E2005}. Here we do not pursue this further
  as it would introduce additional technicalities and deviate the
  attention from the principle question of convergent estimators; see
  also Remark \ref{rem:eps:err:const:sampling} below. Another relevant
  aspect when generating observations via discretizing an SDE is the
  numerical stability of the discretization method; see, e.g.,
  \cite{Kloeden1992,Milstein2004,Buckwar2011}.  However, as a method's
  numerical stability is problem dependent and since we work under the
  assumption that the observations (i.e.\ the data) are given, we will
  not address this topic here further. Instead, we consider
  ``sufficiently small'' time step sizes in Assumption
  \ref{assumption:discretization:weak:convergence:order}, so that no
  stability issues are present.  Finally, we remark that these
  considerations do not apply for real world observations.
\end{remark}

Next we make an assumption on the mean squared convergence of the
approximations of expectations, which is a well-established error
criterion for moment approximations.
\begin{assumption}
  \label{assumption:approx:expectation:order}
  For any $\varphi\in C_b(\R^d)$, $\tau\in\mathcal{T}_h$, and $\xi\in\Xi$, let
  $\bar{u}_{h,N}^\varepsilon(\tau,\xi;\varphi)$ be an approximation of
  $\bar{u}_h^\varepsilon(\tau,\xi;\varphi) :=
  \E\Bigl(\varphi\bigl(\bar{X}_{h\vert\xi}^\varepsilon(\tau)\bigr)\Bigr)$
  such that
  \begin{equation*}
    \E{\Bigl(\bigl(\bar{u}_{h,N}^\varepsilon(\tau,\xi;\varphi) - \bar{u}_h^\varepsilon(\tau,\xi;\varphi)\bigr)^2\Bigr)} \le C N^{-\gamma}\;,
  \end{equation*}
  for some $\gamma>0$. For $\varepsilon ,h$ sufficiently small, both
  $\gamma$ and the constant $C$ are independent of $\varepsilon, h,
  \tau$, and $N$.
\end{assumption}

Finally we impose some temporal regularity on the expectations.
\begin{assumption}
  \label{assumption:approx:integral:order}
  For any $\varphi\in C_b(\R^d)$ and any $\xi\in\Xi$, the function
  $t\mapsto \E\bigl(\varphi\bigl(X_\xi(t)\bigr)\bigr) \equiv
  u(t,\xi;\varphi)$ is H{\"o}lder continuous on $[0,t]$, $t\in[0,T]$,
  with exponent $\alpha >0$.
\end{assumption}

Based on theses strengthened assumptions it is possible to obtain the
following result concerning convergence rates. For convenience we only
present the case where the matrix $A$ satisfies the rank condition
$\rank{(A)} = n$. The case $\rank{(A)} = \min{(m,n)}$ can be treated
similarly.
\begin{proposition}
  \label{prop:conv:rate}
  Let $X$ be the solution to \eqref{eq:sde} corresponding to the true
  parameter $\theta\in\Theta$ in \eqref{eq:parameterization}.
  Moreover, let $\Xi$ and $\phi\in V_n\cap C_P^{2(\beta+1)}(\R^d)$,
  with $\beta$ as in Assumption
  \ref{assumption:discretization:weak:convergence:order}, be such that
  $\rank{(A)} = n$. Furthermore, let $X^\varepsilon$ be a weak
  perturbation of $X$ such that
  \begin{equation*}
    \sup_{t\in[0,T]}\abs{\E\Bigl(\varphi\bigl(X_\xi^\varepsilon(t)\bigr)\Bigr)-\E\Bigl(\varphi\bigl(X_\xi(t)\bigr)\Bigr)}\le C\varepsilon\;,
  \end{equation*}
  for any $\varphi\in C_b(\R^d)$ with $C$ independent of
  $\varepsilon$, and such that Assumptions
  \ref{assumption:discretization:weak:convergence:order},
  \ref{assumption:approx:expectation:order}, and
  \ref{assumption:approx:integral:order} hold. Then for any
  $t\in[0,T]$, it holds with probability exceeding $p\in[0,1[$ that
  \begin{equation}
    \frac{\norm{\hat\theta_{\delta,h,N}^\varepsilon-\theta}_2}{\norm{\theta}_2} \le C\biggl( \varepsilon + \delta^\alpha  + \min{\bigl(1, c(\varepsilon) h^\beta\bigr)} + \frac{N^{-\gamma/2}}{\sqrt{1-p}} \biggr)\;,\label{eq:conv:rate}
  \end{equation}
  for $\varepsilon ,\delta , h$, and $N^{-1}$ sufficiently
  small. Therein the constant $C$ is independent of $\varepsilon$,
  $\delta$, $h$, $p$, and $N$, while the constant $c(\varepsilon)$
  only depends on $\varepsilon>0$.
\end{proposition}
\begin{proof}
  We fix $t\in[0,T]$ and $0\le p<1$.  In view of Chebyshev's
  inequality, Assumption \ref{assumption:approx:expectation:order}
  implies that $\vert\bar{u}_h^\varepsilon(\tau,\xi;\varphi) -
    \bar{u}_{h,N}^\varepsilon(\tau,\xi;\varphi)\vert \le C N^{-\gamma/2}
  /\sqrt{1-p}$ with probability exceeding $p$, for any
  $\tau\in\mathcal{T}_h$, $\xi\in\Xi$, $\varphi\in V_n$. As $\phi\in
  V_n\cap C_P^{2(\beta+1)}(\R^d)$, it follows from the assumptions and
  from \eqref{eq:perturbation:rhs:estimate} that
  \begin{equation*}
    \norm{b - b_{h,N}^\varepsilon}_2
    \le C_b\biggl( \varepsilon  + \min{\bigl(1, c(\varepsilon) h^\beta\bigr)} + \frac{N^{-\gamma/2}}{\sqrt{1-p}} \biggr)
  \end{equation*}
  with probability exceeding $p$, where $C_b$ is independent of
  $\varepsilon$, $h$, $p$, and $N$. The constant $c(\varepsilon)$ only
  depends on $\varepsilon>0$; cf.\ Remark
  \ref{rem:eps:err:const:sampling} below. Similarly, it follows from
  \eqref{eq:perturbation:matrix:estimate} with some algebra that there
  exists a constant $C_A$, independent of $\varepsilon$, $h$,
  $\delta$, $p$, and $N$, such that
\begin{equation*}
    \norm{A - A_{h,\delta,N}^\varepsilon}_2
    \le C_A\biggl( \varepsilon + \delta^\alpha + \min{\bigl(1, c(\varepsilon) h^\beta\bigr)} + \frac{N^{-\gamma/2}}{\sqrt{1-p}} \biggr)
  \end{equation*}
  with probability exceeding $p$ in view of the hypotheses and
  \cite[Thm.\ $1.1$]{Cruz-Uribe2002}. For $\varepsilon$, $\delta$,
  $h$, and $N^{-1}$ sufficiently small, the claim then follows in view
  of \cite[Thm.\ $1.4.6$]{Bjorck1996}.
\end{proof}
\begin{remark}
  \label{rem:eps:err:const:sampling}
  In Proposition \ref{prop:conv:rate} above we use $c(\varepsilon)$ to
  indicate that the error constant in \eqref{eq:conv:rate} could
  depend on $\varepsilon$, due the dependency of the discrete time
  observations in \eqref{eq:discr:weak:conv:order} on the parameter
  $\varepsilon$ (see also Remark \ref{rem:eps:err:const}). It is worth
  mentioning however, that this error contribution due to inexact
  sampling is often neglected in the (statistical) analysis of
  estimation procedures for diffusion processes (see, e.g.\
  \cite{PrakasaRao1999}) and it is instead assumed that the process is
  sampled exactly. When overlooking this particular error contribution
  here too, the convergence rate \eqref{eq:conv:rate} simplifies, as
  $c(\varepsilon)\equiv 0$ in this case.
\end{remark}

Observe that the ensemble estimator \eqref{eq:approx:exp:ensemble} to
approximate expectations is covered by the hypotheses of Proposition
\ref{prop:conv:rate}. In fact, Assumption
\ref{assumption:approx:expectation:order} holds with $\gamma = 1$ in
this case. The situation is more intricate for estimators based on one
long trajectory (i.e.\ time series). This is due to the fact that the
techniques for proving the mean squared convergence of
\eqref{eq:approx:exp:mixing} rely on Taylor expansions of the
stationary density function of the underlying random
variables. Consequently, the error constant in Assumption
\ref{assumption:approx:expectation:order} depends on (partial)
derivatives of this density in this case; see \cite[Thm.\
$3.1$]{Bosq1998}. Therefore, it is not possible to obtain uniform
bounds with respect to the parameters $\varepsilon$ and $h$ as
required by Assumption \ref{assumption:approx:expectation:order}.
From a practical point of view we believe, however, that bound
\eqref{eq:conv:rate} for estimator \eqref{eq:approx:exp:mixing} is
nonetheless useful, here in the form
\begin{equation*}
  \frac{\norm{\hat\theta_{\delta,h,N}^\varepsilon-\theta}_2}{\norm{\theta}_2} \le C\biggl( \varepsilon + \delta^\alpha  + \min{\bigl(1, c_1(\varepsilon) h^\beta\bigr)} + c_2(\varepsilon,h)\frac{N^{-\gamma/2}}{\sqrt{1-p}} \biggr)\;,
\end{equation*}
with $\gamma = 4/(d+4)$, because it highlights the interplay of the
parameters that influence the accuracy and can thus guide numerical
experiments.

Finally, we remark that in practice the combination of discrete time
observations and numerical integration naturally links $h$ and
$\delta$. That is, the numerical integration time step $\delta$ (and
hence $n_\delta$) is not arbitrary but has to be such that
$\delta = lh$ (or $n_t=ln_\delta$), for some $l\in\N$. The choice
$l>1$ could then make sense to reduce the computational effort during
the integral approximation, while bound \eqref{eq:conv:rate} also
suggests to choose $\delta\propto h^{\beta/\alpha}$ so that both error
contributions are of the same order.

%
%
\section{Application: Data-Driven Coarse-Graining for Multiscale Diffusion
  Processes}
\label{sec:numerics}

As motivated in the introduction, one important class of problems for
which it is essential to have a convergent estimation procedure, is
the problem of finding effective coarse-grained systems associated
with the resolved degree of freedom of a multiscale diffusion
process. Specifically, we consider the following prototypical system
of SDEs,
\begin{subequations}
  \begin{align}
    dX^\varepsilon &= \Bigl(\frac{1}{\varepsilon}f_0(X^\varepsilon,Y^\varepsilon) + f_1(X^\varepsilon,Y^\varepsilon) \Bigr)\,dt + \alpha_0(X^\varepsilon,Y^\varepsilon)\,dU_t + \alpha_1(X^\varepsilon,Y^\varepsilon)\,dV_t\;,\label{eq:sde:generic:fastslow:slow}\\
    dY^\varepsilon &= \Bigl(\frac{1}{\varepsilon^2}g_0(X^\varepsilon,Y^\varepsilon) + \frac{1}{\varepsilon}g_1(X^\varepsilon,Y^\varepsilon) + g_2(X^\varepsilon,Y^\varepsilon) \Bigr)\,dt + \frac{1}{\varepsilon}\beta(X^\varepsilon,Y^\varepsilon)\,dV_t\;,\label{eq:sde:generic:fastslow:fast}
\end{align}
\label{eq:sde:generic:fastslow}%
\end{subequations}%
with $X^\varepsilon\colon[0,T]\rightarrow\R^d$ and
$Y^\varepsilon\colon[0,T]\rightarrow\R^{d'}$ for a finite time
interval $[0,T]$. Furthermore
$f_i\colon\R^d\times\R^{d'}\rightarrow\R^d$, $i\in\{0,1\}$,
$\alpha_0\colon\R^d\times\R^{d'}\rightarrow\R^{d\times p}$, and
$\alpha_1\colon\R^d\times\R^{d'}\rightarrow\R^{d\times q}$ as well as
$g_i\colon\R^d\times\R^{d'}\rightarrow\R^{d'}$, $i\in\{0,1,2\}$, and
$\beta\colon\R^d\times\R^{d'}\rightarrow\R^{d'\times q}$.  In
\eqref{eq:sde:generic:fastslow}, $U$ and $V$ denote independent
Brownian motions of dimensions $p$ and $q$, respectively, and
$\varepsilon>0$ is a small parameter. The main goal of data-driven
coarse-graining then is to use only observations of the resolved
degrees of freedom, i.e.\ of $X^\varepsilon$ solving
\eqref{eq:sde:generic:fastslow:slow}, to determine a coarse-grained
process $X$ solving
\begin{equation}
  dX = f(X)\,dt + g(X)\,dW_t\;,\label{eq:sde:generic:effective}
\end{equation}
which approximately retains the essential statistical properties of
$X^\varepsilon$ for $\varepsilon\ll 1$.  This strategy can be made
rigorous using homogenization theory; see \cite[Ch.\ $11$ and
$18$]{Pavliotis2008book} and the references therein for details. In
fact, it is well-known that the process $X^\varepsilon$ solving
\eqref{eq:sde:generic:fastslow:slow} converges weakly in
$C([0,T],\R^d)$ to $X$ solving \eqref{eq:sde:generic:effective},
provided that the fast process $Y^\varepsilon$ is ergodic and the
centering condition is satisfied. That is, $X^\varepsilon$ is a weak
perturbation of $X$ in the sense of Definition
\ref{def:input:perturbation} so that data-driven coarse-graining
corresponds precisely to the problem of estimating parameters in the
SDE \eqref{eq:sde:generic:effective} based on a perturbed input.

Here, we present several data-driven coarse-graining examples to
illustrate the applicability of the estimation methodology described
in Section \ref{sec:estimator}. Although the following examples are
fairly simple, they are yet very instructive as they cover many
different and important aspects, including space dependent
coefficients and multivariate processes. Most importantly, however,
all examples are such that the theoretical results presented in
Section \ref{sec:conv} apply and also so that commonly used
statistical techniques, such as the maximum likelihood estimator, fail
to obtain accurate approximations of the parameters in the
coarse-grained model. We also emphasise that we only use
homogenization theory to construct the weakly convergent process
$X^\varepsilon$ and its limit process $X$ in these numerical examples,
so that we can measure the error of the estimated values and compare
it with the theoretical results in Section \ref{sec:conv}. In fact,
the developed estimation procedure itself does not rely on any
homogenization techniques at all. Moreover, it does neither rely on
the statistical knowledge of $Y^\varepsilon$, i.e.\ knowledge of
\eqref{eq:sde:generic:fastslow:fast}, nor on any other information of
\eqref{eq:sde:generic:fastslow}, even $\varepsilon$ is not assumed to
be known.

If not stated otherwise, the discretely sampled observations were
obtained by solving the multiscale SDE numerically via the
Euler--Maruyama scheme (i.e.\ $\beta = 1$ in Assumption
\ref{assumption:discretization:weak:convergence:order} \cite[Ch.\
$9.1$]{Kloeden1992}) using a time step $h=10^{-3}$, which is
sufficiently small to avoid numerical instabilities in the
observations for the values of $\varepsilon$ considered
below. Therefore, the error due to approximating the solution to the
SDE will be negligible so that we can solely focus on the effect due
to the weak perturbations. Moreover, the temporal subdivision used for
the trapezoidal operator \eqref{eq:quadrature:trapezoidal} to
approximate time integrals is set to equate with the sampling time,
i.e.\ $\delta = h$. The set of trial points $\Xi$ used in the examples
below is a collection of normally distributed random variables, which
were drawn a priori and then fixed throughout the numerical
experiment; see \cite{Kalliadasis2015} for alternative, fully
data-driven, strategies.  Based on these approximations and using only
time discrete observations of $X^\varepsilon$, the goal is to infer
the coefficients in the corresponding coarse-grained model
\eqref{eq:sde:generic:effective}, when assuming that both the drift
function $f$ and the diffusion function $G=gg^T$ can be parametrized
as in \eqref{eq:parameterization}. Recall that the estimated value
depends on the choice of the admissible function $\phi$, the set of
trial points $\Xi$, and the time $t$. Consequently, the error
constants in \eqref{eq:conv:rate} also depend on those parameters, in
particular the dependency of the estimated value on $t$ is
profound. We will thus plot the relative errors of the estimated
values as functions of $t$ below.

\subsection{Fast Ornstein--Uhlenbeck noise}
As a first example, consider the two-dimensional multiscale system
\begin{subequations}
\begin{align}
  dX^\varepsilon &= \Bigl(\frac{1}{\varepsilon}\sigma(X^\varepsilon)Y^\varepsilon + h(X^\varepsilon,Y^\varepsilon) - \sigma'(X^\varepsilon)\sigma(X^\varepsilon)\Bigr)\,dt\;,
  \label{eq:sub:OU:fastslow:slow}\\
  dY^\varepsilon &= -\frac{1}{\varepsilon^2}Y^\varepsilon\,dt +
  \frac{\sqrt{2}}{\varepsilon}\,dV_t\;,\label{eq:sub:OU:fastslow:fast}
\end{align}
\label{eq:sub:OU:fastslow}%
\end{subequations}
for some functions $h\colon\R\times\R\to\R$ and
$\sigma\colon\R\to\R_\ge$, and with $V$ being a standard
one-dimensional Brownian motion. Since the fast process is an
Ornstein--Uhlenbeck process, determining the precise form of a
coarse-grained equation associated to this multiscale system reduces
to computing Gaussian integrals. In fact, the associated
coarse-grained model is given by
\begin{equation}
  dX = \bar{h}(X)\,dt + \sqrt{2\sigma(X)^2}\,dW_t\;,\label{eq:sub:OU:effective}
\end{equation}
where $\bar{h}(x)$ denotes the average of $h(x,\cdot)$ with respect to
the invariant measure of the fast process $Y^\varepsilon$, and $W$
denotes another standard one-dimensional Brownian motion. In
\eqref{eq:sub:OU:fastslow:slow} we have subtracted the Stratonovich
correction from the drift so that the noise in
\eqref{eq:sub:OU:effective} can be interpreted in It\^o's sense. This
drift correction was merely done for convenience and is not essential
for what follows. In the sequel we consider two different choices of
the pair $h(\cdot ),\sigma(\cdot )$. 

As a first example let
\begin{equation}
  h(x,y) = Ax\quad\text{and}\quad\sigma(x) = \sqrt{\varsigma}\;,\label{eq:sub:OU:effective:OU}
\end{equation}
for some $\varsigma \ge 0$, so that \eqref{eq:sub:OU:effective} is the
SDE satisfied by an Ornstein--Uhlenbeck process. Consequently, to fit
\eqref{eq:sub:OU:effective} to available data, we seek $n=2$
parameters. Natural choices for the functions in the drift and
diffusion parametrization \eqref{eq:parameterization} are
\begin{equation*}
f_1(x) = x\;,\quad f_2(x) = 0\;,\quad G_1(x) = 0\;,\quad G_2(x) = 2\;,
\end{equation*}
with the true parameters being
$\theta \equiv (\theta_1,\theta_2)^T = (A,\varsigma)^T$. We chose
$\phi(x)= \exp(-x^2/2)$ as admissible function, and approximate the
expectations by an ensemble average of trajectories. Finally, we
consider $m=24$ different trial points. For the numerical experiment
we generate observations of $X^\varepsilon$ on $[0,t]$, i.e.\ of
\eqref{eq:sub:OU:fastslow:slow}, with $(A,\varsigma) = (-0.5, 0.5)$
and $\varepsilon = 0.1$ in \eqref{eq:sub:OU:fastslow}, and fit the
coarse-grained SDE model \eqref{eq:sub:OU:effective} to these
data. Figure \ref{figure:fastou:ou} depicts the relative errors of the
resulting parameter estimates as a function of time $t$, for two
different ensemble sizes $N\in\{100,5000\}$.
\begin{figure}[t]
  \centering
  \subfigure[ensemble size $N = 5000$]{
    \includegraphics[width=0.465\textwidth]{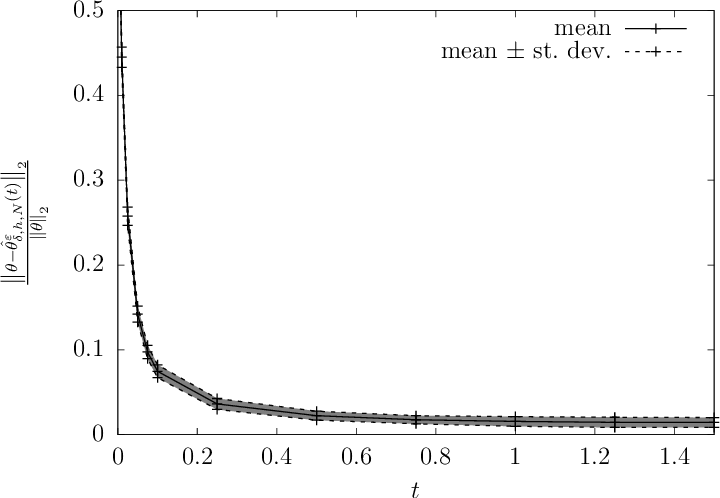}
    \label{fig:fastou:ou:errnorm}
  }
  \quad
  \subfigure[ensemble size $N = 100$]{
    \includegraphics[width=0.465\textwidth]{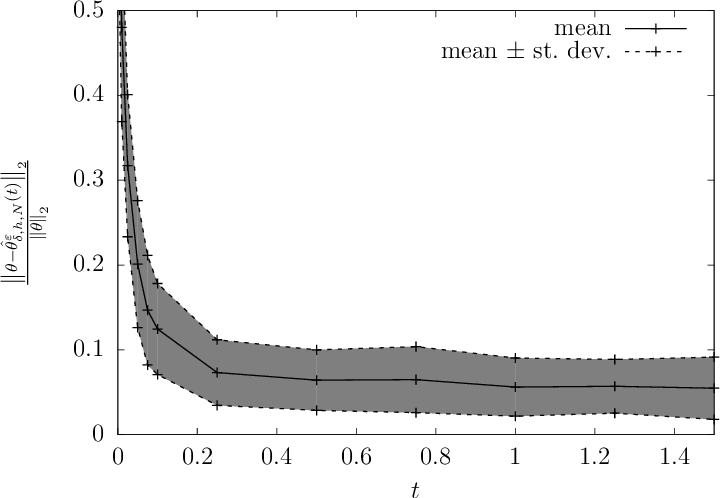}
    \label{fig:fastou:ou:errnorm:small}
  }
  \caption[Relative error of estimator for coarse-grained
  Ornstein--Uhlenbeck SDE]{Mean and standard deviation of the
    estimator's relative error for \eqref{eq:sub:OU:effective} with
    \eqref{eq:sub:OU:effective:OU} as functions of $t$, with
    $h=\delta = 10^{-3}$, and $\varepsilon = 10^{-1}$, for two
    different ensemble sizes $N$.}
  \label{figure:fastou:ou}
\end{figure}
To focus solely on the influence of the input perturbations, i.e.\ to
verify the $\varepsilon$-stability of the methodology numerically, we
plot the relative error in Figure \ref{fig:fastou:ou:errnorm} for a
large ensemble size $N=5000$, so that all other error contributions
are negligible. Specifically, we show the empirical mean and standard
deviation of estimator's relative errors, obtained by repeating the
numerical experiment $100$ times. For very small values of $t$, one
observes large relative errors indicating that the estimators, based
on these approximations, are distorted. In view of
\eqref{eq:conv:rate} this is due to a large constant dominating the
error. Increasing $t$, however, reduces the relative error
significantly, i.e.\ the error constant shrinks. In fact, the mean
relative error drops well below $5\%$ for $t\ge 0.2$ with only minor
fluctuations, indicated by the small standard deviation. Roughly
speaking, by increasing $t$ one increases the information content that
is available to the estimator and the $\mathcal{O}(\varepsilon)$
contribution in error bound \eqref{eq:conv:rate} becomes visible. As a
matter of fact, the formal calculations in \cite[Ch.\
3.4]{Krumscheid2014_phd} for a related estimator suggest that the
multiscale error should even be $\mathcal{O}(\varepsilon^2)$ for this
toy example, which is also confirmed by the numerical experiments
here. In fact, the estimator's mean relative error is between $1\%$
and $2\%$ for $t\ge 0.75$, with standard deviations smaller than
$6\cdot 10^{-3}$. To demonstrate the usefulness of bound
\eqref{eq:conv:rate}, despite the fact that it is rather pessimistic,
Figure \ref{fig:fastou:ou:errnorm:small} illustrates the mean and
standard deviation of the estimator's relative errors for the same
experiment but with a smaller ensemble size $N$. By decreasing $N$,
one can significantly reduce the computational cost while still
controlling the relative error. Specifically, we use $N=100$ so that
$1/\sqrt{N} = \mathcal{O}(\varepsilon)$, which in view of bound
\eqref{eq:conv:rate} should yield relative errors of the same order,
with (possibly) larger fluctuations. As we have however seen in the
previous experiment, the multiscale error is actually
$\mathcal{O}(\varepsilon^2)$ for this simple example, so that we now
expect the estimator's relative error to be dominated by the
statistical error $1/\sqrt{N}= \mathcal{O}(\varepsilon)$.  This is
indeed confirmed in Figure \ref{fig:fastou:ou:errnorm:small}. In fact,
one finds qualitatively the same behaviour as before: the estimator is
biased for small values of $t$, while increasing $t$ considerably
reduces the mean relative error below $7\%$ with some fluctuations
(standard deviation is smaller than $4\cdot 10^{-2}$).

Consider as a second example $h(x,y) = Ax+Bx^3$ and
$\sigma(x) = \sqrt{\sigma_a + \sigma_b x^2}$, so that the
coarse-grained system \eqref{eq:sub:OU:effective} associated with the
multiscale system \eqref{eq:sub:OU:fastslow} reads
\begin{equation}
  dX = (AX+BX^3)\,dt + \sqrt{2(\sigma_a + \sigma_b X^2)}\,dW_t\;.\label{eq:sub:OU:effective:ls4}
\end{equation}
In this case, natural choices for the functions in
\eqref{eq:parameterization} with $n=4$ parameters are
\begin{align*}
  f_1(x)& = x\;,\quad f_2(x) = x^3\;,\quad f_3(x) = 0\;,\quad f_4(x) = 0\;,\\ 
  G_1(x)& = 0\;,\quad G_2(x) = 0\;,\quad G_3(x) = 2\;,\quad G_4(x) = 2x^2\;, 
\end{align*}
where the true parameters are
$\theta \equiv (\theta_1,\theta_2, \theta_3,\theta_4)^T =
(A,B,\sigma_a,\sigma_b)^T$. As admissible function we select
$\phi(x)= (1+x)\exp(-x^2/2)$ to meet the rank condition in Proposition
\ref{prop:asymp:unbiased}, and we approximate the expectations by an
ensemble average again. Finally, we consider $m=54$ trial
points. Figure \ref{figure:fastou:ls4} depicts the 
\begin{figure}[t]
  \centering
  \subfigure[ensemble size $N = 5000$]{
    \includegraphics[width=0.465\textwidth]{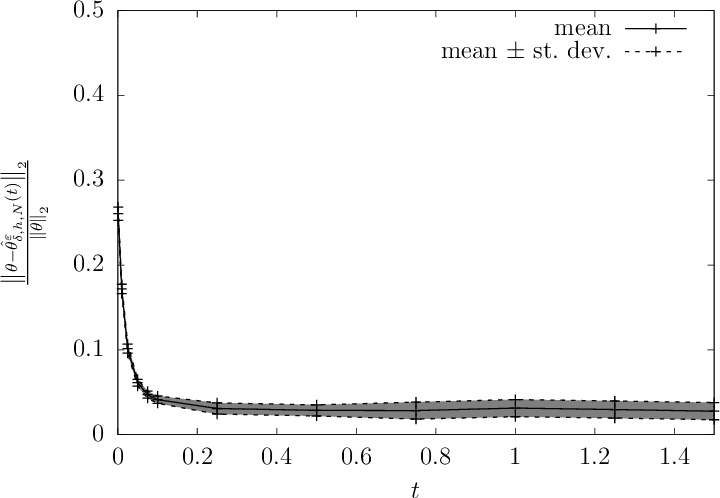}
    \label{fig:fastou:ls4:errnorm}
  }
  \quad
  \subfigure[ensemble size $N = 100$]{
    \includegraphics[width=0.465\textwidth]{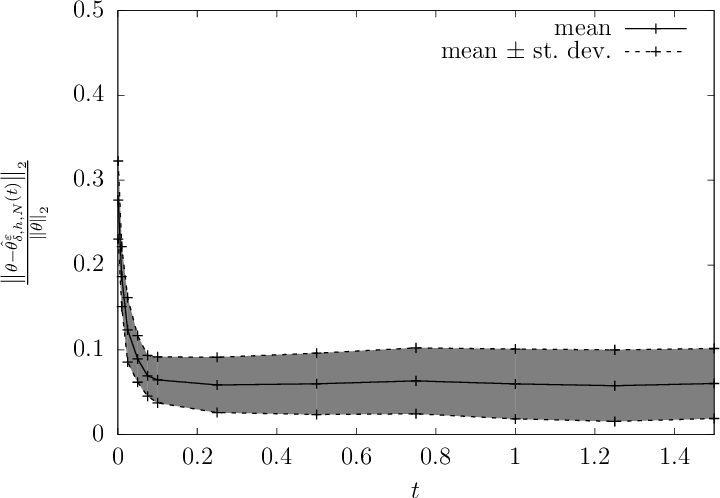}
    \label{fig:fastou:ls4:errnorm:small}
  }
  \caption[Relative error of estimator for coarse-grained
  Landau--Stuart SDE]{Mean and standard deviation of the estimator's
    relative error for \eqref{eq:sub:OU:effective:ls4} as functions of
    $t$, with $h=\delta = 10^{-3}$, and $\varepsilon = 10^{-1}$, for
    two different ensemble sizes $N$.}
  \label{figure:fastou:ls4}
\end{figure}
mean and standard deviation of the estimated value's relative error
for the parameters in \eqref{eq:sub:OU:effective:ls4} corresponding to
the choice $(A,B,\sigma_a,\sigma_b) = (3,-2,3/4,1/2)$ and
$\varepsilon = 0.1$ in \eqref{eq:sub:OU:fastslow}. We consider again
two different ensemble sizes $N\in\{100,5000\}$ and compute the
empirical mean and standard deviations by repeating the same
experiment $100$ times. Despite the fact that SDE
\eqref{eq:sub:OU:effective:ls4}, which models a meta-stable system for
this parameter choices, provides a far more involved structure than
the previous example, the estimation procedure shows qualitatively the
same performance behaviour as before. For the large ensemble size
$N=5000$, Figure \ref{fig:fastou:ls4:errnorm} also displays that
increasing $t$ reduces the mean relative error substantially (to about
$3\%$) and only minor fluctuations (standard deviation is $10^{-2}$)
are present. Even though the results for this example are slightly
less accurate than the ones for the first example, also this example
shows the validity of the error bound \eqref{eq:conv:rate}, although
the bound also appears to be conservative for this example since
actual multiscale error seems to be $\mathcal{O}(\varepsilon^2)$
rather than $\mathcal{O}(\varepsilon)$.  Furthermore, Figure
\ref{fig:fastou:ls4:errnorm:small} also demonstrates the practicality
of bound \eqref{eq:conv:rate} for this example: by decreasing the
ensemble size to $N = 100$, so that
$1/\sqrt{N}=\mathcal{O}(\varepsilon)$, one observes mean relative
errors that show qualitatively the same behaviour as a function of $t$
and that are of now dominated by the statistical error
$1/\sqrt{N}=\mathcal{O}(\varepsilon)$, due to the fact that multiscale
error is $\mathcal{O}(\varepsilon^2)$ here. Moreover, the fluctuations
are comparable to the ones observed in previous example (standard
deviation is smaller than $4.5\cdot 10^{-2}$).
  
\subsection{Brownian motion in a two-dimensional potential}
Another example that falls into the class of multiscale diffusion
processes is the movement model of Brownian motion in a two-scale
potential. Specifically, consider the two-dimensional Langevin
equation
\begin{equation*}
  dX^\varepsilon = -\nabla V\biggl(X^\varepsilon,\frac{1}{\varepsilon}X^\varepsilon;M\biggr)\,dt + \sqrt{2\sigma}\,dU_t\;,
\end{equation*}
where $V(\cdot,\cdot;M)$ denotes a two-scale potential with $M$ being
a set of parameters controlling $V$ and $U$ denotes a standard
two-dimensional Brownian motion used to model the thermal noise. We
assume that the two-scale potential $V(\cdot,\cdot;M)$ is given by a
large scale as well as a separable fluctuating part
$V(x,y;M) = V(x;M) + p_1(y_1) + p_2(y_2)$, with
$x,y\equiv (y_1,y_2)^T\in\R^2$, so that the original system reads
\begin{equation}
  dX^\varepsilon = -\Biggl(\nabla V(X^\varepsilon;M) +
  \frac{1}{\varepsilon}\begin{pmatrix}p_1'\bigl(X_1^\varepsilon/\varepsilon\bigr)\\p_2'\bigl(X_2^\varepsilon/\varepsilon\bigr)\end{pmatrix} \Biggr)\,dt + \sqrt{2\sigma}\,dU_t\;,\label{eq:sub:brown:multi2d}
\end{equation}
where $X^\varepsilon(t) \equiv
{\bigl(X_1^\varepsilon(t),X_2^\varepsilon(t)\bigr)}^T\in\R^2$. Here we
take the large scale part to be a quadratic potential $V(x;M) =
\tfrac{1}{2}x^TMx$, with $M\in\R^{2\times 2}$ being symmetric and
positive definite, so that the coarse-grained equation for
$X(t)\in\R^2$ is given by
\begin{equation}
  dX = -RMX\,dt + \sqrt{2\sigma R}\,dW_t\;,\label{eq:sub:brown:effective2d}
\end{equation}
with analytic expressions for $R=\diag(r_1,r_2)$
\cite{Pavliotis2007}. For the choices $p_1(y_1) = \cos(y_1)$ and
$p_2(y_2) = \cos(y_2)/2$ we find that $r_1 = I_0(1/\sigma)^{-2}$ and
$r_2= I_0(1/(2\sigma))^{-2}$, where $I_0(z)$ denotes the modified
Bessel function of the first kind. For this example a simple choice
for the functions defining the parametrization
\eqref{eq:parameterization} is
\begin{equation*}
  f_1(x) = \begin{pmatrix}x_1\\0\end{pmatrix}\;,\quad f_2(x) = \begin{pmatrix}x_2\\0\end{pmatrix}\;,
  \quad f_3(x) = \begin{pmatrix}0\\x_1\end{pmatrix}\;,\quad f_4(x) = \begin{pmatrix}0\\x_2\end{pmatrix}\;,
  \end{equation*}
  and $f_5(x) = f_6(x) = 0$, as well as $G_1(x) = G_2(x) = G_3(x) = G_4(x) = 0$ and
\begin{equation*} 
   G_5(x) = \begin{pmatrix}2&0\\0&0\end{pmatrix}\;,\quad G_6(x) = \begin{pmatrix}0&0\\0&2\end{pmatrix}\;,
\end{equation*}
where $x\equiv {(x_1,x_2)}^T\in\R^2$.  Hence, we seek to determine $n=6$
parameters, where the true parameters $\theta \equiv (\theta_1,\dots
,\theta_6)^T$ are such that
$\bigl(\begin{smallmatrix}\theta_1&\theta_2\\\theta_3&\theta_4\end{smallmatrix}\bigr)
= -RM$ and $\diag(\theta_5,\theta_6) = \sigma R$. As admissible
function we select here $\phi(x) = \Phi(x_1)\Phi(x_2)$, with
$\Phi(z) = (1+z^2)\exp(-z^2/2)$. Moreover, we choose $m=24$ trial points and
approximate the expectations by ensemble averages ($N=5000$). Figure
\ref{figure:langevin2d} shows the relative error of the
\begin{figure}[t]
    \centering
    \includegraphics[width=0.465\textwidth]{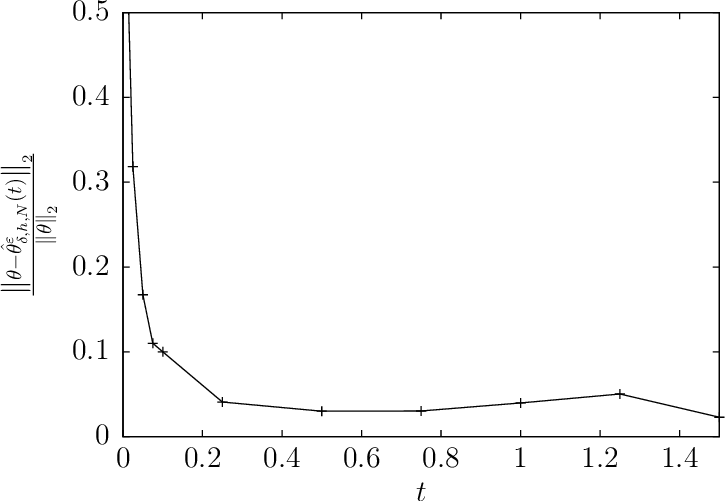}
    \caption[Relative error of estimator for Brownian motion in 2d
    potential]{Relative error of the estimators
      $\hat{\theta}_{\delta,h,N}^\varepsilon$ for
      \eqref{eq:sub:brown:effective2d} as functions of $t$,
      with ensemble size $N=5000$, $h=\delta = 10^{-3}$, and
      $\varepsilon = 10^{-1}$.}
    \label{figure:langevin2d}
\end{figure}
estimated value as a function of $t$ based on observations of the
multiscale system \eqref{eq:sub:brown:multi2d} with $M =
\bigl(\begin{smallmatrix}2&2\\2&3\end{smallmatrix}\bigr)$,
$\sigma=3/2$, and $\varepsilon=0.1$. Also here we observe that the
relative error is significantly reduced to around $5\%$ by increasing
$t$ and only minor fluctuations are present.

\subsection{Brownian motion in a two-scale potential revisited}
In the previous examples we always used an ensemble average to
approximate the expectations. Here we illustrate that the proposed
methodology can also be applied to the situation where only one long
trajectory of observations (i.e.\ a time series) is
available. Consider the one-dimensional Langevin equation
\begin{equation*}
  dX^\varepsilon = -\frac{d}{dx} V_\alpha\Bigl(X^\varepsilon,\frac{1}{\varepsilon} X^\varepsilon\Bigr)\,dt + \sqrt{2\sigma}\,dU_t\;.
\end{equation*}
Let the two-scale potential $V_\alpha$ be given by a quadratic large
scale part plus a fluctuating part, $V_\alpha(x,y) = \alpha x^2/2 +
p(y)$, so that the Langevin equation can be written as
\begin{equation}
  dX^\varepsilon = -\Bigl(\alpha X^\varepsilon + \frac{1}{\varepsilon} p'\bigl(X^\varepsilon/\varepsilon\bigr)\Bigr)\,dt + \sqrt{2\sigma}\,dU_t\;. \label{eq:eps:sde:langevin1d}
\end{equation}
When the fluctuating part $p$ is sufficiently smooth, bounded, and
periodic with period $L$, the coarse-grained equation is given by
\begin{equation}
  dX = -A X\,dt + \sqrt{2\Sigma}\,dW_t\;,\label{eq:eps:sde:langevin1d:coarse}
\end{equation}
with $A=\alpha L^2/(Z_{+}Z_{-})$ and $\Sigma=\sigma L^2/(Z_{+}Z_{-})$,
where $Z_{\pm} = \int_0^Le^{\pm p(y)/\sigma}\,dy$.

To effectively use the estimation procedure based on one long
trajectory of time discrete approximations of
\eqref{eq:eps:sde:langevin1d}, the time discrete approximations have
to satisfy a mixing condition, as detailed in Proposition
\ref{prop:conv:expect:mixing}. To check this condition, we assume that
the time discrete approximations are the result of an Euler--Maruyama
approximation and that $p'$ is bounded. Let
$g_h^\varepsilon(x) := (1-\alpha h)x -
p'(x/\varepsilon)h/\varepsilon$, then the Euler--Maruyama scheme
applied to \eqref{eq:eps:sde:langevin1d} on $[0,t=n_th]$ can be
written as
\begin{equation}
  \bar{X}_{h\vert\xi}^\varepsilon\bigl((k+1)h\bigr) = g_h^\varepsilon\bigl( \bar{X}_{h\vert\xi}^\varepsilon(kh) \bigr) + \eta_k\sqrt{2\sigma h}\;,\quad\bar{X}_{h\vert\xi}^\varepsilon(0) = \xi\;,\label{eq:eps:sde:langevin1d:EM}
\end{equation}
for $0\le k< n_t$, where the sequence of random variables
${(\eta_k)}_{0\le k< n_t}$ is i.i.d.\ with
$\eta_0\sim\mathcal{N}(0,1)$. For any $h,\varepsilon > 0$ sufficiently
small, one can thus find $b,c>0$ and $a\in(0,1)$ such that
$\abs{g(x)}\le a\abs{x}-b$ for $\abs{x}\ge c$, since $p'$ is
bounded. As the Euler--Maruyama scheme
\eqref{eq:eps:sde:langevin1d:EM} generates essentially a stochastic
difference equation of autoregressive type, it follows from \cite[p.\
$102$]{Doukhan1994} that the process
${\bigl(\bar{X}_{h\vert\xi}^\varepsilon(kh)\bigr)}_{k\ge 0}$ is
strictly stationary and geometrically $\alpha$-mixing. Consequently,
Proposition \ref{prop:conv:expect:mixing} ensures that the error of
approximating the expectation by the regression estimator
\eqref{eq:approx:exp:mixing} vanishes and that the main convergence
result (Proposition \ref{prop:asymp:unbiased}) holds.

To estimate the $n=2$ parameters in
\eqref{eq:eps:sde:langevin1d:coarse}, we use $f_1(x) = x$,
$f_2(x) = 0 = G_1(x)$, and $G_2(x) = 2$ in the parametrization
\eqref{eq:parameterization}. For the numerical experiment below we set
$p(y)=\cos(y)$ so that the true parameters are
$\theta \equiv (\theta_1,\theta_2)^T =
{I_0(\sigma^{-1})}^{-2}(-\alpha,\sigma)^T$, with $I_0(z)$ again
denoting the modified Bessel function of first kind. Moreover, we use
$m=24$ trial point and $\phi(x) = \exp(-x^2/2)$ as admissible
function.
\begin{figure}[t]
    \centering
    \includegraphics[width=0.465\textwidth]{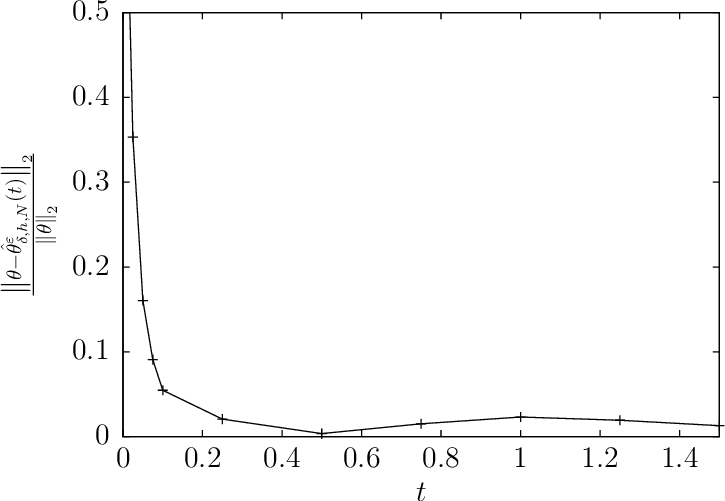}
    \caption[Relative error of estimator for Brownian motion in 1d
    potential]{Relative error of the estimators
      $\hat{\theta}_{\delta,h,N}^\varepsilon$ for
      \eqref{eq:eps:sde:langevin1d:coarse} as functions of
      $t$, using $h=\delta = 10^{-3}$, $\varepsilon = 10^{-1}$, and a
      single time series on $[0,5000]$.}
    \label{figure:langevin1d}
\end{figure}
Figure \ref{figure:langevin1d} shows the relative error of the
estimated value as a function of $t$, when one trajectory of
observations on $[0,5000]$ is obtained from the multiscale system with
$(\alpha,\sigma) = (2,1)$, and $\varepsilon = 0.1$. The same behaviour
of the relative error as a function of $t$ is evident: very small $t$
yields distorted estimated values, while increasing $t$ reduces the
error significantly. In fact, for $t\ge 0.1$ the relative error drops
well below $5\%$ with only minor fluctuations. Since bound
\eqref{eq:conv:rate} is not guaranteed to be valid in this case, the
constants in front of the rates might depends on other parameters (see
discussion in Section \ref{conv:errana:rates}). Therefore we chose a
rather long time series to focus solely on $\varepsilon$-stability,
that is on the influence of the perturbation of the input, and to
illustrate the convergent behaviour of the estimation procedure.

%
%
\section{Conclusion}
\label{sec:conclusion}
We have studied the convergence of parametric estimation procedures
for diffusion processes from a numerical analysis
perspective. Specifically, we have introduced consistency, stability,
and convergence concepts for estimation procedures. It turns out that
the maximum likelihood estimator is not convergent within this
framework, since it fails to be stable. Conversely, we have introduced
an inference methodology which is provably convergent within this
framework. This convergence property of an estimation procedure is
pivotal in many applications, such as for data-driven coarse-graining
approaches from multiscale observations. We have studied several
examples of this class to verify the theoretical results of the
introduced methodology. Furthermore, these examples demonstrate that
the estimation procedure can be used to accurately approximate
parameters in both the drift function and the diffusion function.

There are still many challenges that remain to be addressed. One is,
for example related to the rigorous verification of the mixing
conditions in the case where only one time series is available. From a
theoretical perspective this is not easy, as the available theory is
quite restrictive. In fact, most of it is only applicable for a
constant diffusion coefficient and a drift satisfying a linear growth
condition; see, e.g., \cite{Klokov2013} and references
therein. Standard conditions on drift and diffusion functions ensuring
the mixing conditions of the continuous time diffusion process are,
e.g., given in
\cite{Veretennikov1987,Veretennikov1989,Leblanc1997}. From a practical
perspective, however, this condition does not appear to be too
restrictive, as the results in \cite{Kalliadasis2015} indicate.

But there are also other interesting questions left open. During the
construction of the estimator, for example, there are still some
degrees of freedom, which we have not used optimally yet. For
instance, it seems that the particular choice of the admissible
function $\phi$ can influence the error constant of the error
bound. Therefore, an important task for future research is to study
whether or not one can minimise the error constant not only with
respect to $\phi$, but also with respect to the number and location of
the trial points. Moreover, characterising the error constant's
dependency on the parameter $t$ is also desirable. A closely related
avenue for future efforts is also the study of the asymptotic
distribution of the estimators, which in turn can be used to guide the
construction of asymptotic confidence intervals for the estimated
values. These and related topics will be treated in future studies.

\section*{Acknowledgements}
I would like to thank both anonymous referees for their insightful
comments and suggestions. Furthermore, I am grateful to my former PhD
supervisors Prof.\ G.A.~Pavliotis and Prof.\ S.~Kalliadasis for many
useful comments and suggestions. Thanks are also due to Dr.\
A.~Veraart and Prof.\ S.~Reich for critically reading an earlier
version of the manuscript and their helpful comments. This work was
supported by the Engineering and Physical Sciences Research Council of
the UK through Grant No.\ EP/H034587.

\bibliography{refs4paper}

\bibliographystyle{siam}

\end{document}